\documentclass[12pt]{amsart}

\usepackage[margin=1.2in]{geometry}

\usepackage{amssymb}
\usepackage[cmtip,all]{xy}
\usepackage{url}
\usepackage[bookmarksnumbered,colorlinks,bookmarks,citecolor=blue,linkcolor=blue,urlcolor=blue,breaklinks,linktocpage]{hyperref}

\newtheorem{theorem}{Theorem}[section]
\newtheorem{lemma}[theorem]{Lemma}
\newtheorem{proposition}[theorem]{Proposition}
\newtheorem{corollary}[theorem]{Corollary}

\theoremstyle{definition}
\newtheorem{definition}[theorem]{Definition}
\newtheorem{example}[theorem]{Example}

\newtheorem{remark}[theorem]{Remark}

\theoremstyle{remark}

\numberwithin{equation}{section}

\newcommand{\co}{\colon}
\newcommand{\ob}{\operatorname{Ob}}
\newcommand{\mor}{\operatorname{Mor}}

\newcommand{\cala}{{\mathcal A}}
\newcommand{\calb}{{\mathcal B}}
\newcommand{\calc}{{\mathcal C}}
\newcommand{\cald}{{\mathcal D}}

\newcommand{\calf}{{\mathcal F}}

\newcommand{\cali}{{\mathcal I}}

\newcommand{\calp}{{\mathcal P}}

\newcommand{\bfK}{{\mathbf K}}

\begin{document}

\title[Approximation for Quasicategories]{Approximation in $K$-theory for Waldhausen Quasicategories}
\author{Thomas M. Fiore}
\address{Thomas M. Fiore \\ Department of Mathematics and Statistics\\
University of Michigan-Dearborn \\ 4901 Evergreen Road \\ Dearborn,
MI 48128 \\ U.S.A.}
\email{tmfiore@umich.edu}
\urladdr{http://www-personal.umd.umich.edu/~tmfiore/}


\begin{abstract}
We prove a series of Approximation Theorems in the setting of Waldhausen quasicategories. These theorems, inspired by Waldhausen's 1985 Approximation Theorem, give sufficient conditions for an exact functor of Waldhausen quasicategories to induce a level-wise weak homotopy equivalence of $K$-theory spectra.

The Pre-Approximation Theorem, which holds in the general setting of quasicategories {\it without} Waldhausen structures, provides sufficient conditions for a functor $F\co \cala \to \calb$ to restrict to an equivalence of the maximal $\infty$-groupoids in $\cala$ and $\calb$. These conditions are: $F$ reflects equivalences, every codomain morphism $Fa \to b$ factors as $(\text{equiv}) \circ F(f)$, the functor $ho(F)$ of homotopy categories is essentially surjective and full on isomorphisms, the domain quasicategory $\cala$ admits colimits of diagrams of equivalences indexed by connected finite posets, and $F$ preserves such colimits.

Our Approximation Theorems follow from the Pre-Approximation Theorem. The Approximation Theorem in the quasicategorical setting most analogous to Waldhausen's is: if an exact functor $F\co \cala \to \calb$ satisfies Waldhausen's App 1 and App 2, and the domain $\cala$ admits colimits of the aforementioned type and $F$ preserves them, then $\mathbf{K}(F)$ is a level-wise equivalence. As a corollary, if $F$ is an exact functor with $ho(F)$ an equivalence of ordinary categories, and every morphism in the domain $\cala$ is a cofibration, then $\mathbf{K}(F)$ is a level-wise equivalence.

We then introduce a version of App 2 called Cofibration App 2 that only requires factorization of cofibrations $Fa \rightarrowtail b$ as $(\text{equiv}) \circ F(\text{cofibration})$ and prove an analogous Cofibration Approximation Theorem, and a corollary for certain functors that induce an equivalence of cofibration homotopy categories.

We also prove that $S_n^\infty$ is Waldhausen equivalent to $\overline{\calf_{n-1}^\infty}$ using the mid anodyne maps known as spine inclusions, and clarify how hypotheses and notions in Waldhausen structures are related in new ways in the context of quasicategories.
\\ [2mm]


\end{abstract}

\maketitle

\newpage

\tableofcontents

\section{Introduction}

When does a map $F$ induce a stable equivalence of algebraic $K$-theory spectra? Waldhausen's Approximation Theorem \cite[Theorem 1.6.7]{WaldhausenAlgKTheoryI} gives just such a criterion: if $F \co \cala \to \calb$ is an exact functor between reasonable Waldhausen categories, $F$ reflects weak equivalences, and every codomain morphism $Fa \to b$ factors as $(\text{w.e.}) \circ F(\text{cofibration})$, then $F$ induces a stable equivalence of $K$-theory spectra (even $wS_\bullet F$ is a weak homotopy equivalence).

The main purpose of the present paper is to prove the most general Approximation Theorem presently available in the context of quasicategories: if an exact functor $F$ of Waldhausen quasicategories satisfies
\begin{enumerate}
\item
{\it (App 1)} \\
$F$ reflects equivalences,
\item
{\it (App 2)} \\
For every $a \in \cala$ and every morphism $Fa \to b$ in the codomain $\calb$, there exists a cofibration $f \co a \rightarrowtail a'$ in $\cala$, an equivalence $F(a') \simeq b$ in $\calb$, and a 2-simplex in $\calb$ of the form
$$\begin{array}{c}\xymatrix{Fa \ar[rr]^-{\forall} \ar@{>->}[dr]_{\exists \;\; F(f)} & & b \\  & F(a') \ar[ur]_{\exists \;\;\text{\rm equivalence}} & }
\end{array},$$
and
\item
The domain quasicategory $\cala$ admits colimits of diagrams of equivalences indexed by connected finite posets, and $F$ preserves such colimits,
\end{enumerate}
then $F$ induces a level-wise weak homotopy equivalence of $K$-theory spectra. Recent general results in this direction worked with localizations of homotopical categories, or assumed that every map is a cofibration (see the literature review later in this introduction).

A few years after Waldhausen's Approximation Theorem appeared, Thomason--Trobaugh published a proof in
\cite[1.9.8]{ThomasonTrobaugh} that a functor which induces an equivalence of derived homotopy categories induces an equivalence in $K$-theory, under appropriate hypotheses. In 1.9.9, Thomason remarks ``Morally, it says that $K(A)$ essentially depends only on the derived category $w^{-1}A$, and thus that Waldhausen $K$-theory gives essentially a $K$-theory of the derived category.''

Since then, many mathematicians have obtained results in this direction in various contexts: Neeman \cite{NeemanIA, NeemanIB} and Dugger-Shipley \cite{DuggerShipley} for algebraic $K$-theory of rings, Schlichting's counterexample for Frobenius categories \cite{SchlichtingInventiones}, To\"en--Vezzosi \cite{ToenVezzosi} and Blumberg--Mandell \cite{BlumbergMandellAbstHomTheory} for Dwyer-Kan simplicial localizations of Waldhausen categories, and Cisinski \cite{CisinskiInvarianceOfK-Theory} for right exact functors between reasonable Waldhausen categories. Sagave \cite{SagaveSpecialApproximation}, in the context of classical Waldhausen categories, showed how to loosen Waldhausen's requirement of factorizing $Fa \to b$ mentioned above to a requirement of factorizing $Fa \to b$ when $b$ is a ``special object''. In Appendix A of \cite{Schlichting}, Schlichting showed how to replace Waldhausen's cylinder functor and cylinder axiom in Waldhausen's classical Approximation Theorem by requiring factorization of any domain morphism into a cofibration followed by a weak equivalence (the full hypotheses of Schlichting's Approximation variant are: $F$ reflects equivalences, every codomain morphism $Fa \to b$ factors as $(\text{w.e.}) \circ F(\text{cofibration})$, every morphism of $\cala$ factors as $(\text{w.e.}) \circ (\text{cofibration})$, and the weak equivalences in both the domain and codomain satisfy the 3-for-2 property).

The main contribution of the present article is to prove several versions of Waldhausen Approximation in the context of the Waldhausen {\bf quasi}categories of Barwick \cite{Barwick} and Fiore--Pieper \cite{FiorePieper}, see Definition~\ref{def:Waldhausen_quasicat}, Theorem~\ref{thm:Approximation}, Corollary~\ref{cor:Approx_hoF_equiv_and_coA=A}, Theorem~\ref{thm:CofibrationApproximation}, Corollary~\ref{cor:Approximation_when_F_induces_cof_equiv}, and Corollary~\ref{cor:Approximation_when_F_rflcts_cofs_and_hoF_is_equivalence}. All of these follow from the Pre-Approximation Theorem~\ref{thm:Pre-Approximation}, which does not require a Waldhausen structure. We recall the main notion and examples, and describe these results now.

A {\it quasicategory}, or {\it $\infty$-category}, is a simplicial set in which every inner horn has a filler. For instance Kan complexes and nerves of categories are quasicategories. A {\it Waldhausen quasicategory} is a quasicategory with zero objects together with a selected 1-full subquasicategory of cofibrations containing the equivalences. A Waldhausen quasicategory is required to have pushouts along cofibrations, and any morphism with domain any zero object is required to be a cofibration. ``Weak equivalences'' are not part of the definition, instead the ``weak equivalences'' in a Waldhausen quasicategory are by default the {\it equivalences} of the quasicategory, i.e. those morphisms that are invertible in the homotopy category of the quasicategory.\footnote{The 1-full subquasicategory on the equivalences is the maximal Kan subcomplex by \cite[Theorems~4.18 and 4.19]{JoyalQuadern}. Recall also that Kan complexes are often called $\infty$-groupoids.}

We mention here several examples of Waldhausen quasicategories. The nerve of any classical Waldhausen category with weak equivalences the isomorphisms is a Waldhausen quasicategory. For instance, the nerve of the classical Waldhausen category of based finite sets, bijections, and injections is a Waldhausen quasicategory. When a Waldhausen category comes from a model category, then Barwick's {\it relative neve} is a Waldhausen quasicategory with the same $K$-theory, see \cite[Corollary~10.18]{Barwick} and see \cite[Example~4.5]{FiorePieper} for a summary of the related results in \cite[Sections~9 and 10]{Barwick} by Barwick. Genuinely non-classical examples of Waldhausen quasicategories are the stable quasicategories of Lurie (in a stable quasicategory every morphism is considered a cofibration).  Recall that a quasicategory is {\it stable} if it admits all finite limits and colimits, and pushout squares and pullback squares coincide. Stable quasicategories are the context of Lurie's second book \cite{LurieHigherAlgebra}.

The main results of this paper begin with the Pre-Approximation Theorem~\ref{thm:Pre-Approximation}:  if $F\co \cala \to \calb$ is a functor of quasicategories that reflects equivalences, every codomain morphism $Fa \to b$ factors as $(\text{equiv}) \circ F(f)$, the functor $ho(F)$ of homotopy categories is essentially surjective and full on isomorphisms, the domain quasicategory $\cala$ admits colimits of diagrams of equivalences indexed by connected finite posets, and $F$ preserves such colimits, then $F$ restricts to an equivalence of the maximal $\infty$-groupoids of $\cala$ and $\calb$. Notice that the Pre-Approximation Theorem does not require any Waldhausen structures. The idea for the proof of Theorem~\ref{thm:Pre-Approximation} is due to Waldhausen: use Proposition~\ref{prop:extracted_from_Waldhausen} and (quasicategorical) Quillen's Theorem A.  However, we also incorporate a quasicategorical implementation of an idea of Schlichting \cite[page 132]{Schlichting}, see the discussion preceding Theorem~\ref{thm:Pre-Approximation}.

The first consequence of the Pre-Approximation Theorem is the Approximation Theorem~\ref{thm:Approximation}: if $F$ satisfies App 1 and App 2, and $\cala$ admits all finite colimits and $F$ preserves them, then $\mathbf{K}(F)$ is a level-wise equivalence of spectra. The requirement of finite colimits in $\cala$ and their preservation by $F$ could be replaced by the requirement of colimits in $\cala$ of a more specific type and their preservation by $F$, namely colimits of diagrams of equivalences indexed by connected finite posets. We prefer to require ``finite colimits''  because it is easier to state than the aforementioned more specific type of colimit. Corollary~\ref{cor:Approx_hoF_equiv_and_coA=A}, on the other hand, assumes $hoF$ is an equivalence and $co\cala=\cala$ in order to conclude $\mathbf{K}(F)$ is a level-wise equivalence of spectra.

Theorem~\ref{thm:CofibrationApproximation}, Corollary~\ref{cor:Approximation_when_F_induces_cof_equiv}, and Corollary~\ref{cor:Approximation_when_F_rflcts_cofs_and_hoF_is_equivalence} are versions of Approximation in which App 2 is replaced by the more general Cofibration App 2. In other words, instead of requiring every codomain morphism of the form $Fa \to b$ to factor as $(\text{equivalence}) \circ F(\text{cofibration})$, which would imply every morphism $Fa \to b$ is a cofibration, we merely require every cofibration $Fa \rightarrowtail b$ to factor as $(\text{equivalence}) \circ F(\text{cofibration})$, thus allowing non-cofibration morphisms $Fa \to b$ in $\calb$. In Theorem~\ref{thm:CofibrationApproximation}, Corollary~\ref{cor:Approximation_when_F_induces_cof_equiv}, and Corollary~\ref{cor:Approximation_when_F_rflcts_cofs_and_hoF_is_equivalence}, we require the {\it domain cofibration subquasicategory} $co\cala$ to admit colimits of diagrams in $\cala_\mathrm{equiv}$ indexed by connected finite posets, and $F$ is required to preserve such colimits. In Theorem~\ref{thm:CofibrationApproximation} we also require $ho(F_\text{equiv})$ and $ho\big(\big(\overline{\calf^\infty_n}F\big)_\text{\rm equiv}\big)$ to be full. On the other hand, Corollary~\ref{cor:Approximation_when_F_induces_cof_equiv} and Corollary~\ref{cor:Approximation_when_F_rflcts_cofs_and_hoF_is_equivalence} require just for $n=1$ the functor $ho\big(\big(\overline{\calf^\infty_1}F\big)_\text{\rm equiv}\big)$ to be full, because $ho(coF)$ is an equivalence of categories in those corollaries. In Corollary~\ref{cor:Approximation_when_F_rflcts_cofs_and_hoF_is_equivalence} $hoF$ is assumed to be an equivalence of categories, but the assumption that $F$ reflects cofibrations implies that $ho(coF)$ is also an equivalence of categories.

Factorization in the sense that every morphism in $\cala$ or $\calb$ factors as $(\text{equivalence}) \circ (\text{cofibration})$ is not required for the main results (indeed, factorization in the quasicategorical context is equivalent to requiring all maps to be cofibrations, a strong assumption we do not want to make).

Related Approximation results in the quasicategorical literature are the following. Barwick proved the special case of our Approximation Theorem in which all maps are cofibrations. Namely, Barwick's Corollary~8.4 of \cite{Barwick} implies: if $F\co \cala \to \calb$ is an exact functor between Waldhausen quasicategories, both of which have all maps cofibrations and have all finite colimits, and $F$ induces an equivalence $hoF \co ho\cala \to ho\calb$ of homotopy categories, then $F$ induces a stable equivalence of $K$-theory spectra (in fact he shows $F$ induces an equivalence in any ``additive theory''). Blumberg--Gepner--Tabuada proved a stronger statement under the assumption of stability (all maps in a stable quasicategory are considered to be cofibrations). Namely Blumberg--Gepner--Tabuada's Corollary~5.11 of \cite{BlumbergGepnerTabuadaI} states that a map of stable quasicategories is an equivalence if and only if it induces an equivalence of their homotopy categories. In general, Blumberg--Gepner--Tabuada study $K$-theory in \cite{BlumbergGepnerTabuadaI} as an invariant of stable quasicategories. In the present paper, we do not require stability, nor does Barwick.

Simplicial categories, on the other hand, are another model for $\infty$-categories, and recent $K$-theory literature contains some related Approximation results concerning Dwyer-Kan hammock localizations of classical Waldhausen categories. To\"en--Vezzosi \cite{ToenVezzosi} observed already in 2004 that the $K$-theory of a ``good'' category with fibrations and weak equivalences is an invariant of the underlying $\infty$-category, namely of its Dwyer--Kan hammock localization, despite the fact that the $K$-theory cannot be reconstructed from the triangulated homotopy category, as proved by Neeman \cite{NeemanInventiones}. Equivalence of localizations is in fact closely correlated to Approximation. Blumberg-Mandell \cite[Theorems~1.5 and 1.4]{BlumbergMandellAbstHomTheory} proved that equivalence of Dwyer--Kan localizations follows from Waldhausen's Approximation axioms when $\cala$ and $\calb$ satisfy 3-for-2 and factorization, see Section~\ref{subsec:ComparisonWithOtherAuthors} of the present paper. Cisinski \cite{CisinskiInvarianceOfK-Theory} proved that this is actually an ``if and only if'' statement (see Theorems 2.9 and 3.25, Proposition 4.5, and Scholie 4.15, all in \cite{CisinskiInvarianceOfK-Theory}).

Assuming Dwyer-Kan equivalence of Dwyer-Kan localizations of weak cofibration subcategories, and a few other hypotheses, a consequence of \cite{BlumbergMandellAbstHomTheory} is a stable equivalence of $K$-theory spectra, see Remark~\ref{rem:BlumbergMandell_DKequiv_of_weak_cof_cats} for details. The present article remains entirely in the world of Waldhausen quasicategories.

\smallskip

{\bf Outline of Paper.} For readers not already familiar with quasicategories, I review in Appendix Section~\ref{sec:recollections} all the necessary results of Boardman--Vogt, Joyal, and Lurie, including homotopy, join, slice, and colimits. There I also prove a criterion for a simplicial set to be weakly contractible, which is a simplicial version of a categorical result Schlichting extracted from Waldhausen's paper \cite{WaldhausenAlgKTheoryI}. Section~\ref{sec:Motivational_Examples} rapidly presents new examples that are covered by the Cofibration Approximation Theorem~\ref{thm:CofibrationApproximation} but not the Approximation Theorem~\ref{thm:Approximation}. Some terminology from the rest of the paper is used here, but the purpose is to quickly motive the rest of the work. In Section~\ref{sec:Waldhausen Quasicategories}, I recall the notion of Waldhausen quasicategory, discuss some of its consequences, and introduce the variants $\overline{S_\bullet^\infty}$ and $\overline{\calf_\bullet^\infty}$ of Waldhausen's constructions using spines $I[n]$ rather than simplices $\Delta[n]$. Section~\ref{sec:approximation_theorems} is the heart of the paper, it begins with the main technical result of the paper, the Pre-Approximation Theorem~\ref{thm:Pre-Approximation}, then proves various quasicategorical Approximation Theorems, and compares the axioms Pre-App 2, App 2, and Cofibration App 2. Section~\ref{sec:approximation_theorems} also contains a more detailed comparison with other results in the literature.

\section{Examples of Quasicategorical Cofibration Approximation Theorem not Covered by Approximation Theorem} \label{sec:Motivational_Examples}

In this section we provide some motivational examples for the quasicategorical Cofibration Approximation Theorem~\ref{thm:CofibrationApproximation} and the rest of the paper. We freely use terminology that is introduced in the sequel.

The quasicategorical Cofibration Approximation Theorem~\ref{thm:CofibrationApproximation} has genuinely new examples. We begin with the inclusion $co\calc \hookrightarrow \calc$, and then specify it to a homotopy version of based finite sets, called $h\calf in_\ast$. The comparison of  $h\calf in_\ast$ with the classical $\calf in_\ast$ also delivers examples of several results in this paper.

A consequence of the Cofibration Approximation Theorem~\ref{thm:CofibrationApproximation} is the following.
\\ \\
{\bf Corollary}~\ref{cor:cofibration_inclusion}. {\it Let $\calc$ be Waldhausen quasicategory. Then $co\calc \hookrightarrow \calc$ induces a level-wise equivalence in $K$-theory. }
\\ \\
\noindent If $\calc$ has a morphism that is not a cofibration, then the inclusion $co\calc \hookrightarrow \calc$ does not satisfy App 2, and the Approximation Theorem~\ref{thm:Approximation} does not apply, but the Cofibration Approximation Theorem~\ref{thm:CofibrationApproximation} does.

Consider now the quasicategory $h\calf in_\ast$ of based ``homotopy finite, homotopy discrete spaces.'' This quasicategory is the simplicial nerve of the Kan enriched category of based Kan complexes of the form
\begin{equation} \label{equ:object_of_hFinstar}
\coprod_{i=1}^m L_i \hspace{1in} \text{with each $L_i$ a contractible Kan complex.}
\end{equation}
A based simplicial map $f$ between such based Kan complexes is a cofibration in the Waldhausen structure on $h\calf in_\ast$ if $\pi_0f$ is injective. The equivalences in the quasicategory $h\calf in_\ast$ are the based weak homotopy equivalences between such based Kan complexes.

To prove that the equivalences in the quasicategory $h\calf in_\ast$ are the based weak homotopy equivalences between such based Kan complexes, we appeal to the foundational results on classical Quillen model structure on $\mathbf{SSet}_*$ in combination with Boardman and Vogt's results about homotopy in a quasicategory. First recall that a based map in $\mathbf{SSet}_*$ is a cofibration, fibration, or weak equivalence if and only if its underlying unbased map is. Thus, the fibrant objects in $\mathbf{SSet}_*$ are exactly the based Kan complexes, and every object of $\mathbf{SSet}_*$ is cofibrant. Consequently a based map between based Kan complexes is a weak equivalence if and only if it has a {\it based homotopy} inverse. On the other hand, a morphism in the quasicategory $h\calf in_\ast$ is an equivalence if and only if it has a homotopy inverse {\it in the quasicategory} $h\calf in_\ast$. But by \cite[Lemma~1.9]{JoyalQuadern} (attributed to Boardman-Vogt \cite{BoardmanVogt}), homotopy between morphisms in the quasicategory $h\calf in_\ast$ is the same as based homotopy of based maps between based Kan complexes. Thus, the equivalences in the quasicategory $h\calf in_\ast$ are the based weak homotopy equivalences between based Kan complexes of the form \eqref{equ:object_of_hFinstar}.

The quasicategory $h\calf in_\ast$ is a Waldhausen quasicategory in the sense of Definition~\ref{def:Waldhausen_quasicat}, but {\it not} a Waldhausen category: ordinary pushouts in the {\it category} of such based Kan complexes \eqref{equ:object_of_hFinstar} do not exist. For instance, the following pushout in $\mathbf{SSet}$ of 3 such based\footnote{We take $1'$ to be the basepoint in the top two objects, and $1$ to be the basepoint in the bottom right object.} Kan complexes \eqref{equ:object_of_hFinstar} is not a Kan complex.
\begin{equation} \label{equ:hFinstar_category_not_closed}
\begin{array}{c}
\xymatrix@R=3pc@C=5pc{\{1', 1'' \} \ar@{^{(}->}[r] \ar[d] \ar@{}[dr]|{\text{strict p.o. in category $\mathbf{SSet}$}} & N(0\cong 1') \coprod N(1'' \cong 2) \ar[d] \\ \ast \ar[r] & \left[0 \rightleftarrows 1 \rightleftarrows 2\right]}
\end{array}
\end{equation}
The two left simplicial sets $\{1', 1''\}$ and $\ast$ are finite and discrete, hence are homotopy finite and homotopy discrete Kan complexes. The upper right simplicial set $N(0\cong 1') \coprod N(1'' \cong 2)$ is the disjoint union of two nerves of contractible groupoids, so is a homotopy finite and homotopy discrete Kan complex. The top map is a trivial cofibration in the Quillen model structure on $\mathbf{SSet}$, i.e. it is a mono weak homotopy equivalence so the pushout is even a homotopy pushout in $\mathbf{SSet}$. The top map is of course also $\pi_0$-injective, so is a cofibration in the Waldhausen structure on the simplicial nerve $h\calf in_\ast$.

The bottom right corner pushout object we have indicated only with its non-degenerate 1-skeleton $\left[0 \rightleftarrows 1 \rightleftarrows 2\right]$. It is not a Kan complex because the inner horn $0 \to 1 \to 2$ does not have a filler, so it is not even a quasicategory.

Since the pushout object in \eqref{equ:hFinstar_category_not_closed} is not of the form \eqref{equ:object_of_hFinstar}, the {\it category} of based Kan complexes of the form \eqref{equ:object_of_hFinstar} is not a Waldhausen category.

However, the quasicategory $h\calf in_\ast$ does admit colimits along its cofibrations. Namely, since $h\calf in_\ast$ is the simplicial nerve of a Kan-enriched simplicial category, homotopy colimits in the simplicial category are the same as the colimits in the quasicategory $h\calf in_\ast$ by \cite[Theorem~4.2.4.1]{LurieHigherToposTheory}. We can form homotopy pushouts along $\pi_0$-injective maps of based Kan complexes of the form \eqref{equ:object_of_hFinstar} similar to how pushouts along injective based maps in $\calf in_\ast$ can be formed.

Thus $h\calf in_\ast$ is a genuine example of a Waldhausen quasicategory, and the inclusion
$$\xymatrix{ co(h\calf in_\ast)  \ar@{^{(}->}[r]  &   h\calf in_\ast}$$
induces a level-wise equivalence in $K$-theory by the Cofibration Approximation Theorem, but {\it not} by the Approximation Theorem.

We may now also consider the following inclusion functor and path component functor.
\begin{equation} \label{equ:i_and_pi0}
\xymatrix{i \colon N\calf in_\ast \ar@{^{(}->}[r] & h\calf in_\ast } \hspace{1in} \xymatrix{\pi_0\colon h\calf in_\ast \ar[r] & N\calf in_\ast }
\end{equation}

The inclusion functor $i$ does not satisfy App 2. Namely, to satisfy App 2, the non-injective constant $c$ map would have to factor as a composite of injections.
$$\xymatrix@R=4pc{i\{a,b\} \ar[rr]^-{\text{constant}_c} \ar@{>..>}[dr]_{\exists ? \;\; i(\text{based injective})\hspace{.25in}} & & \{c\} \\  & \exists ? \;\; i(\text{based finite set}) \ar@{..>}[ur]_{\hspace{.25in} \exists?  \;\;\text{\rm equivalence=bijection}} & }$$
But that is impossible, so $i$ does not satisfy App 2. But $i$ does satisfy Cofibration App~2. Any given $\pi_0$-injective map in $h\calf in_\ast$
$$\xymatrix{g \colon i(\text{based finite set}) \ar@{>->}[r] & L}$$
with codomain a based Kan complex $L$ of the form \eqref{equ:object_of_hFinstar} factors as
$$\xymatrix@R=4pc@C=4pc{i(\text{based finite  set}) \ar[rr]^-{\text{given $\pi_0$-injective map }g} \ar@{>->}[dr]_{\exists \;\; i(\text{based injective})\hspace{.25in}} & & L \\  & i(\pi_0 L) \ar@{>->}[ur]_{\hspace{.1in} \exists  \;\;\text{\rm equivalence}} & }$$
where the based injective map sends a point to the path component it lands in, and the equivalence sends a path component in $\pi_0L$ to the point in that path component in the image of $g$. So $i$ satisfies Cofibration App 2.

By similar considerations, one can show that $\pi_0$ in \eqref{equ:i_and_pi0} does not satisfy App 2 but does satisfy Cofibration App 2.

Both $i$ and $\pi_0$ in \eqref{equ:i_and_pi0} satisfy the hypotheses of the Cofibration Approximation Theorem~\ref{thm:CofibrationApproximation}, Corollary~\ref{cor:Approximation_when_F_induces_cof_equiv}, and Corollary~\ref{cor:Approximation_when_F_rflcts_cofs_and_hoF_is_equivalence}, so they each induce a level-wise equivalence of Waldhausen $K$-theory spectra.

We should not be surprised $h\calf in_\ast$ and $N\calf in_\ast$ have the same $K$-theory: $i$ and $\pi_0$ in \eqref{equ:i_and_pi0} are, after all, inverse Waldhausen equivalences of Waldhausen quasicategories in the sense of \cite[Definition~7.8 and Proposition~7.9]{FiorePieper}. But it is still interesting to note that the Approximation Theorem~\ref{thm:Approximation} does not apply, though Cofibration Approximation~\ref{thm:CofibrationApproximation} does.

Further examples of non-equivalences that satisfy the hypotheses of the Cofibration Approximation Theorem~\ref{thm:CofibrationApproximation} and Corollary~\ref{cor:Approximation_when_F_induces_cof_equiv} are the composites

\begin{equation}
\xymatrix{co(h\calf in_\ast) \ar@{^{(}->}[r] & h\calf in_\ast \ar[r]^{\pi_0} & N\calf in_\ast }
\end{equation}
\begin{equation}
\xymatrix{co(\calf in_\ast) \ar@{^{(}->}[r] & N\calf in_\ast \ar@{^{(}->}[r]^i & h\calf in_\ast }.
\end{equation}
These composites satisfy Cofibration App 2 because both second morphisms do and both first morphisms are essentially surjective, see Proposition~\ref{prop:compositions_satisfy_App2}~\ref{prop:compositions_satisfy_App2:hyp:FE}.

\section{Waldhausen Quasicategories} \label{sec:Waldhausen Quasicategories}

We first recall some background on Waldhausen quasicategories from the paper of Fiore--Pieper \cite{FiorePieper}, discuss consequences of the definition, and introduce a variant $\overline{\calf_n^\infty}\calc$ of Waldhausen's $\calf_n\calc$. Recall that a subquasicategory $A$ of a quasicategory $X$ is {\it 1-full} if any simplex of $A$ is in $X$ if and only if all of its edges are in $X$.

\subsection{Waldhausen Quasicategories and Exact Functors} \label{subsec:Review_of_WaldhausenQCats} \leavevmode \\

\begin{definition}[Waldhausen quasicategory, \cite{FiorePieper}] \label{def:Waldhausen_quasicat}
A {\it Waldhausen quasicategory} consists of a quasicategory $\calc$ with zero objects and a subquasicategory $co \calc$, the 1-simplices of which are called {\it cofibrations} and denoted $\rightarrowtail$, such that
\begin{enumerate}
\item \label{def:Waldhausen_quasicat:(i)}
The subquasicategory $co \calc$ is 1-full in $\calc$ and contains all equivalences in $\calc$,
\item \label{def:Waldhausen_quasicat:(ii)}
For each object $A$ of $\calc$ and any zero object $\ast$ of $\calc$, every morphism $\ast \to A$ is a cofibration,
\item \label{def:Waldhausen_quasicat:(iii)}
The pushout of a cofibration along any morphism exists, and every pushout of a cofibration along any morphism is a cofibration.
\end{enumerate}
\end{definition}

Barwick's notion of Waldhausen $\infty$-category in \cite[Definition~2.7]{Barwick} is equivalent to Definition~\ref{def:Waldhausen_quasicat}. See Fiore--Pieper \cite[Proposition~4.3]{FiorePieper} for a proof.

Though Definition~\ref{def:Waldhausen_quasicat} looks much like the classical definition, there are some important differences which have far-reaching consequences. Perhaps most prominently, a classical Waldhausen category comes equipped with a class of ``weak equivalences'' which contains the isomorphisms, whereas a Waldhausen {\it quasi}category has its class of ``weak equivalences'' pre-selected as the {\it equivalences} of the underlying quasicategory. These are exactly the maps which become isomorphisms in the homotopy category.

Consequently, extra hypotheses on the equivalences (typically needed in a discussion of Approximation) are rarely needed because they automatically hold. For instance, the 3-for-2 property for equivalences in a quasicategory $X$ follows immediately from the 3-for-2 property for isomorphisms in the homotopy category $\tau_1X$ and the fact that 2-simplices in $X$ give rise to commutative triangles in $\tau_1X$. Even more strongly, the equivalences in a quasicategory $X$ satisfy the 6-for-2 property\footnote{A class of maps has the {\it 6-for-2 property} if whenever we have any three composable maps $\overset{u\;}{\rightarrow}\overset{v\;}{\rightarrow}\overset{w\;}{\rightarrow}$ with $vu$ and $wv$ in the class, we can conclude that $u$, $v$, $w$, and $wvu$ are also in the class. The 6-for-2 property implies the 3-for-2 property.} by a similar argument applied to $Sk^2\Delta[3] \to X$ using the 6-for-2 property of the isomorphisms in $\tau_1X$ (the isomorphisms in any category satisfy the 6-for-2 property).

Another consequence of choosing the equivalences as the ``weak equivalences'' is that any map homotopic to a cofibration is also a cofibration, see \cite[Section 4]{FiorePieper}.

We also know that $\tau_1 (co \calc)$ is naturally a subcategory of $\tau_1 \calc$ which contains the isomorphisms of $\tau_1 \calc$. Namely from the 1-fullness of $co\calc$ in $\calc$ in Definition~\ref{def:Waldhausen_quasicat}, it follows that for any cofibrations $f,g,h$ there is a 2-simplex in $co\calc$ with boundary $(g,h,f)$ if and only if there is a 2-simplex in $\calc$ with boundary $(g,h,f)$.

The choice of the ``weak equivalences'' as the equivalences also changes the way that classical hypotheses are related to one another. For instance, the factorization axiom in a Waldhausen quasicategory is equivalent to the requirement that every map is a cofibration. A Waldhausen quasicategory $\calc$ is said to {\it admit factorization} if for any morphism $f$ of $\calc$ there exists a 2-simplex $\sigma$ with boundary $$\partial \sigma=(d_0\sigma,d_1\sigma,d_2\sigma)=(\text{ equivalence }, f, \text{ cofibration }).$$
As a proof that factorization is equivalent to all maps being cofibrations, if $\calc$ admits factorization, and $f$ is any morphism in $\calc$, then
$[f]=[w][c]$ in the homotopy category for some equivalence $w$ and some cofibration $c$. But every equivalence is a cofibration, and $\tau_1(co\calc)$ is naturally a subcategory of $\tau_1\calc$, so $[f]$ is the homotopy class of a cofibration. But any map homotopic to a cofibration is a cofibration, so $f$ is a cofibration. The converse is clear.

Every ``weak cofibration'' in a Waldhausen {\it quasi}category is actually a cofibration, so we make no distinction. More precisely, in \cite{BlumbergMandellActa} and \cite{BlumbergMandellAbstHomTheory}, Blumberg--Mandell call a morphism $f$ in a {\it classical} Waldhausen category a {\it weak cofibration} if there is a zig-zag of weak equivalences in the arrow category from $f$ to a cofibration. If we have a commutative square in a Waldhausen quasicategory $\calc$ (that is a map $\Delta[1] \times \Delta[1] \to \calc$)
\begin{equation} \label{equ:diagram_for_homotopy_repleteness}
\begin{array}{c}
\xymatrix{x \ar[r]^r \ar[d]_{\text{equiv}}^u & y \ar[d]^{\text{equiv }}_v \\ x' \ar[r]_{r'} & y'}
\end{array}
\end{equation}
then in the homotopy category $\tau_1\calc$ we have $$[r]=[v^{-1}][r'][u]\;\;\;\ \text{and} \;\;\;[r']=[v][r][u^{-1}],$$
so $[r]$ is in $\tau_1(co\calc)$ if and only if $[r']$ is in $\tau_1(co\calc)$ (recall that $\tau_1(co\calc)$ contains the isomorphisms of $\tau_1\calc$). But the homotopy class $[r]$ is in $\tau_1(co\calc)$ if and only if the morphism $r$ is in $co \calc$, and similarly for $[r']$ and $r'$. So in any vertical zig-zag of commutative squares like \eqref{equ:diagram_for_homotopy_repleteness} in a Waldhausen quasicategory, with all the vertical arrows weak equivalences, if any one of the horizontal arrows is a cofibration, then all of the horizontal arrows are cofibrations, and weak cofibrations are cofibrations.

Similarly, a {\it homotopy cocartesian square} in the sense of \cite{BlumbergMandellActa} and \cite{BlumbergMandellAbstHomTheory}, but in a Waldhausen {\it quasi}category, is merely a pushout square in which one of the legs is a cofibration (the opposite morphism will then also be a cofibration). This follows from the fact that pushouts in a quasicategory are invariant under equivalence.

After this discussion of the far-reaching consequences of the definition of Waldhausen quasicategory, we can now return to the review of the theory itself.

\begin{definition}[Exact functor]
Let $\cala$ and $\calb$ be Waldhausen quasicategories.
A functor $F\colon \cala \to \calb$ is called {\it exact} if it sends zero objects of $\cala$ to zero objects of $\calb$, it sends cofibrations to cofibrations, and maps each pushout square along a cofibration to a pushout square along a cofibration.
\end{definition}

By 1-fullness of the subquasicategory of cofibrations, we have $F(co\cala) \subseteq co \calb$ for the entire cofibration sub quasicategory $co\cala$. Every map of quasicategories sends equivalences to equivalences, so we also have $F(\cala_\text{equiv}) \subseteq \calb_\text{equiv}$ by 1-fullness of the maximal sub Kan complex. By the comments preceding the definition, an exact functor between Waldhausen quasicategories preserves weak cofibrations and homotopy cocartesian squares.

\subsection{The $S_\bullet^\infty$ Construction} \label{subsec:S} \leavevmode \\

We next recall the $S_\bullet^\infty$ construction. See also \cite[1.2.2.2 and 1.2.2.5]{LurieHigherAlgebra} of Lurie, and \cite[Section~7.2]{BlumbergGepnerTabuadaI} where Blumberg--Gepner--Tabuada compare $S_\bullet^\infty$ with the $S_\bullet'$ construction of \cite{BlumbergMandellActa} and \cite{BlumbergMandellAbstHomTheory} in the case of a simplicial model category which admits all finite homotopy colimits.\footnote{In \cite[Corollary 7.7]{BlumbergGepnerTabuadaI}, Blumberg--Gepner--Tabuada consider a simplicial model category $\cala$ and a small full sub simplicial category $\calc$ which has all finite homotopy colimits, apply the $S_\bullet^\infty$ construction of Lurie to the simplicial nerve of the full sub simplicial category on the fibrant-cofibrant objects of $\calc$, and compare it with the nerve of $wS'_\bullet$. Blumberg--Gepner--Tabuada primarily work with Waldhausen quasicategories in which all maps are cofibrations, so they do not usually explicitly mention cofibrations. For instance, every map in a model category is a weak cofibration, and in a Waldhausen quasicategory weak cofibrations are the same as cofibrations, so in \cite[Corollary 7.7]{BlumbergGepnerTabuadaI} there is no need to specify that the horizontal maps in the $S^\infty_\bullet$ construction applied to a model category are cofibrations.} See also Fiore--Pieper \cite{FiorePieper}. For a fibrational version of $S_\bullet^\infty$, see Barwick \cite[Section 5]{Barwick}.

Let $\text{Ar}[n]$ be the {\it category of arrows} in $[n]$. It is the partially ordered set with elements $(i,j)$ such that $0 \leq i \leq j \leq n$, and with the order $(i,j) \leq (i',j')$ whenever $i \leq i'$ and $j \leq j'$. The category $\text{Ar}[n]$ is an upper-triangular subgrid of an $n \times n$-grid of squares. The upper triangular subgrid $\text{Ar}[n]$ contains the main diagonal of objects, but none below it.

\begin{definition}[$S_\bullet^\infty$ Construction] \label{def:Sbulletinfinity}
Let $\mathcal{C}$ be a Waldhausen quasicategory. An {\it $[n]$-complex} is a map of simplicial sets $A \co N\text{Ar}[n] \to \mathcal{C}$ such that
\begin{enumerate}
\item \label{def:Sbulletinfinity:i}
For each $j \in [n]$, $A(j,j)$ is a zero object of $\mathcal{C}$, possibly different for each $j$,
\item \label{def:Sbulletinfinity:ii}
For each $i \leq j \leq k$, the morphism $A(i,j) \to A(i,k)$ is a cofibration,
\item \label{def:Sbulletinfinity:iii}
For each $i \leq j \leq k$, the diagram
$$\xymatrix{A(i,j) \ar@{>->}[r] \ar[d] & A(i,k) \ar[d] \\ A(j,j) \ar@{>->}[r] & A(j,k)}$$
is a pushout square in $\mathcal{C}$.
\end{enumerate}
Let $S_n^\infty \calc$ be the 0-full sub simplicial set of $\text{Map}(N\text{Ar}[n],\mathcal{C})$ on the $[n]$-complexes.\footnote{In \cite{LurieHigherAlgebra} and \cite{BlumbergGepnerTabuadaI}, which have all maps cofibrations, $S^{\infty}_n\mathcal{C}$ is also called $\mathrm{Gap}([n], \calc)$.} The {\it $S_\bullet^\infty$ construction} of $\mathcal{C}$ is the simplicial quasicategory $S_\bullet^\infty \mathcal{C}$ defined by
$$S_n^\infty \mathcal{C}\subseteq \mathbf{SSet}(N\left(\text{Ar}[n]\right) \times \Delta[-],\mathcal{C}).$$
The objects of $S_n^\infty\mathcal{C}$ are sequences of cofibrations in $\calc$
\begin{equation} \label{equ:sequence_of_cofibrations}
\xymatrix{\ast \ar@{>->}[r] & A_{0,1} \ar@{>->}[r] & A_{0,2} \ar@{>->}[r] & \cdots  \ar@{>->}[r] & A_{0,n}}
\end{equation}
with a choice of quotient $A_{0,j}/A_{0,i} := A_{i,j}$ for each $i \leq j$, and a choice of all composites with simplices that fill them.
The face and degeneracy maps of the simplicial object $S_\bullet^\infty \mathcal{C}$ in the category $\mathbf{QCat}$ are induced from $\Delta$ via the category of arrows construction. In particular this means, if $A \co N\text{Ar}[n] \to \mathcal{C}$ and $\alpha\co [k] \to [n]$, then $\alpha^* A=A \circ (\alpha_* \times \alpha_*)\co N\text{Ar}[k] \to \mathcal{C}$.
\end{definition}

Each quasicategory $S_n^\infty\calc$ is a Waldhausen quasicategory. A morphism $f\co A \to B$ in $S_n^\infty\mathcal{C}$ is a {\it cofibration in $S_n^\infty\mathcal{C}$} if each $f_{0,j}$ is a cofibration {\it and} each pushout morphism
$$A_{0,j} \cup_{A_{0,j-1}} B_{0,j-1} \to B_{0,j}$$ is a cofibration in $\mathcal{C}$ for each $1 \leq j \leq n$. If a morphism is level-wise homotopic to a cofibration in $S_n^\infty\mathcal{C}$, then it is also a cofibration in $S_n^\infty\mathcal{C}$. A natural transformation $f\co A \to B$ is an equivalence in $S_n\mathcal{C}$ if and only if each component $f_{i,j}$ is an equivalence in $\calc$ by \cite[Theorem 5.14]{JoyalQuadern} or \cite[Corollary 3.7]{FiorePieper}. Because all the squares in an object are pushouts, this is equivalent to requiring every $f_{0,j}\co A_{0,j} \to B_{0,j}$ to be an equivalence in $\mathcal{C}$ (the pushout of an equivalence in a quasicategory along any morphism is also an equivalence, see \cite[Lemma~3.23]{FiorePieper}). See \cite[8.3.15]{RognesTextbook} for the case of Waldhausen categories.

\begin{definition} \label{def:nth-K-theory-space}
The {\it $n$-th $K$-theory space} is
$$\bfK_n(\calc):= \big| \left(S_\bullet^\infty \cdots S_\bullet^\infty \calc \right)_\text{equiv}\big|$$
for $n\geq0$, where $S_\bullet^\infty$ appears $n$ times.
\end{definition}

The $K$-theory spaces form an $\Omega$-spectrum $\bfK(\calc)$ beyond the 0-th term.

\subsection{The $\overline{S_\bullet^\infty}$ Construction, or the $S_\bullet^\infty$ Construction without Composites} \leavevmode \\

An important difference between categories and quasicategories motivates our variant of the classical definition: a sequence of $n$ morphisms in a category, head to tail, already has uniquely determined composites of all head-to-tail subseqences and their composites, whereas in a {\it quasi}category such composites for a head-to-tail sequence exist but are not chosen. For a category $\cald$ for instance, if $I[n]$ denotes the {\it spine} of the $n$-simplex $\Delta[n]$, namely the 1-dimensional subcomplex of $\Delta[n]$ given by the union of the edges $[i-1,i]$ for $1\leq i \leq n$, then we have
$$\mathbf{SSet}(I[n],N\cald)\cong\mathbf{Cat}(\tau_1 I[n], \cald) \cong \mathbf{Cat}(\tau_1 \Delta[n], \cald)\cong\mathbf{SSet}(\Delta[n],N\cald).$$
In other words, indicating a sequence of $n$ morphisms in a category $\cald$ is tantamount to indicating all of the composites of subsequences (this is one reason why it is customary in category theory to only draw the $n$ morphisms to indicate the $n$ morphisms {\it along with} all of their composites).

To see how such composites of subsequences exist in a quasicategory $X$, notice that a map $I[n] \rightarrow X$ is a sequence of $n$ morphisms in $X$, while a map $\Delta[n] \rightarrow X$ is a sequence of $n$ morphisms in $X$ together with a choice of composites for all possible subsequences and their composites, together with simplices which make them commute.
The
{\it spine inclusion} $I[n] \hookrightarrow \Delta[n]$ is mid anodyne \cite[Proposition 2.13]{JoyalQuadern} so that the following lift exists.
$$\xymatrix{I[n] \ar[r] \ar[d]_{\text{mid anodyne}} & X \ar[d]^{\text{mid fibration}} \\ \Delta[n] \ar[r] \ar@{-->}[ur] & \ast }$$

Nevertheless, there is little theoretical difference in considering maps $I[n] \to X$ versus maps $\Delta[n] \to X$ because for any quasicategory $X$, the spine inclusion induces an equivalence $X^{\Delta[n]} \to X^{I[n]}$ by \cite[Propositions 2.27 and 2.29]{JoyalQuadern}. Practically, it is easier to work with $I[n]$, and this is what we do now for $\overline{S_\bullet^\infty}$, and for $\overline{\calf_\bullet^\infty}$ in Section~\ref{subsec:F}.

\begin{definition}[$\overline{S_\bullet^\infty}$ Construction] \label{def:overlineSbulletinfinity}
Let $\overline{N\text{Ar}[n]}$ be the intersection of $N\text{Ar}[n]$ with the 2-dimensional simplicial set $I[n] \times I[n]$.
A {\it restricted $[n]$-complex in a Waldhausen quasicategory $\calc$} is a map of simplicial sets $A \co \overline{N\text{Ar}[n]} \to \mathcal{C}$ such that
\begin{enumerate}
\item \label{def:overlineSbulletinfinity:i}
For each $j \in [n]$, $A(j,j)$ is a zero object of $\mathcal{C}$, possibly different for each $j$,
\item \label{def:overlineSbulletinfinity:ii}
For each $i \leq j$, the morphism $A(i,j) \to A(i,j+1)$ is a cofibration,
\item \label{def:overlineSbulletinfinity:iii}
For each $i \leq j$, the diagram
$$\xymatrix{A(i,j) \ar@{>->}[r] \ar[d] & A(i,j+1) \ar[d] \\ A(i+1,j) \ar@{>->}[r] & A(i+1,j+1)}$$
is a pushout square in $\mathcal{C}$.
\end{enumerate}
Then $\overline{S_n^\infty}\calc$ is the full sub-simplicial set of $\mathrm{Map}(\overline{N\text{Ar}[n]}, \calc)$ on the restricted $[n]$-complexes. The cofibrations in $\overline{S_n^\infty}\calc$ are defined just as in $S_n^\infty\calc$.
\end{definition}

Note that the $\overline{S_\bullet^\infty}$ construction is not a simplicial object.

\begin{remark}[Equivalent descriptions of $S_n^\infty \calc$ and $\overline{S_n^\infty} \calc$] \label{rem:zS}
Recall that $N\text{Ar}[n]$ is the upper triangular part (including main diagonal) of the diagram $\Delta[n] \times \Delta[n]$, and similarly $\overline{N\text{Ar}[n]}$ is the upper triangular part (including main diagonal) of the diagram $I[n] \times I[n]$. We can non-uniquely extend any element of $S_n^\infty \calc$ to a functor $\Delta[n] \times \Delta[n] \to \calc$ which is a zero object on each $(i,j)$ with $i > j$.  Similarly, we can non-uniquely extend any element of $\overline{S_n^\infty} \calc$ to a functor $I[n] \times I[n] \to \calc$ which is $\ast$ on each $(i,j)$ with $i > j$. More precisely, let $zS_n^\infty \calc$ be the sub quasicategory of $\calc^{\Delta[n] \times \Delta[n]}$ that is 0-full on the functors $A\colon \Delta[n] \times \Delta[n] \to \calc$ which are zero objects on each $(i,j)$ with $i > j$ and satisfy \ref{def:Sbulletinfinity:i}, \ref{def:Sbulletinfinity:ii}, and \ref{def:Sbulletinfinity:iii} of Definition~\ref{def:Sbulletinfinity}. Similarly let $\overline{zS_n^\infty} \calc$ be the sub quasicategory of $\calc^{I[n] \times I[n]}$ that is 0-full on the functors $A\colon \Delta[n] \times \Delta[n] \to \calc$ which are zero objects on each $(i,j)$ with $i > j$ and satisfy \ref{def:overlineSbulletinfinity:i}, \ref{def:overlineSbulletinfinity:ii}, and \ref{def:overlineSbulletinfinity:iii} of Definition~\ref{def:overlineSbulletinfinity}.

Then the forgetful functors
$$\xymatrix{zS_n^\infty \calc \ar[r] & S_n^\infty \calc}$$
$$\xymatrix{\overline{zS_n^\infty \calc} \ar[r] & \overline{S_n^\infty} \calc}$$
are equivalences of quasicategories.
\end{remark}

\begin{proposition}[Waldhausen equivalence of $S_n^\infty \calc$ with $\overline{S_n^\infty} \calc$] \label{prop:Sn_equivalent_to_Snbar}
For any Waldhausen quasicategory $\calc$, the forgetful functor $S_n^\infty \calc \to \overline{S_n^\infty} \calc$ induced by the inclusion $\overline{N\mathrm{Ar}[n]}\hookrightarrow N\mathrm{Ar}[n]$ is a Waldhausen equivalence, that is, it is an exact functor which admits an exact pseudo inverse (see \cite[Section~6]{FiorePieper} for more on Waldhausen equivalences).
\end{proposition}
\begin{proof}
Every mid anodyne map is a weak categorical equivalence\footnote{We recall here the notion of {\it weak categorical equivalence} from Joyal \cite{JoyalQuadern}, which Lurie calls ``categorical equivalence'' in \cite{LurieHigherToposTheory}. For any simplicial set $A$, let $\tau_0A$ be the set of isomorphism classes of objects of the category $\tau_1A$ (recall $\tau_1$ is the left adjoint to the nerve). A map of simplicial sets $u \colon A \to B$ is a {\it weak categorical equivalence} if for every quasicategory $X$ the induced map $\tau_0(X^B)\to \tau_0(X^A)$ is bijective. For maps between quasicategories, this coincides with the expected notation of equivalence: a functor between quasicategories is a weak categorical equivalence if and only if it is an equivalence in the 2-category $\mathbf{SSet}^{\tau_1}$, i.e., an equivalence of quasicategories. By \cite[Proposition~2.27]{JoyalQuadern}, a map $u\colon A \to B$ is a weak categorical equivalence if and only the map $X^u \colon X^B \to X^A$ is an equivalence of quasicategories for every quasicategory $X$. The weak categorical equivalences are the weak equivalences in Joyal's model structure on $\mathbf{SSet}$, see \cite[Theorem~6.12]{JoyalQuadern}.} \cite[Corollary~2.29]{JoyalQuadern}, and the Cartesian product of two weak categorical equivalences is a weak categorical equivalence \cite[Proposition~2.28]{JoyalQuadern}, so $I[n] \times I[n] \hookrightarrow \Delta[n] \times \Delta[n]$ is a weak categorical equivalence. Then the induced map $\calc^{\Delta[n] \times \Delta[n]} \to \calc^{I[n] \times I[n]}$ is an equivalence of quasicategories by \cite[Proposition~2.27]{JoyalQuadern}. Let us call this restriction equivalence $G$. Then so far we have the following two commutative squares and information from Remark~\ref{rem:zS}.
$$\xymatrix@C=3pc@R=3pc{\calc^{\Delta[n] \times \Delta[n]} \ar[r]^G_{\text{equiv}}  & \calc^{I[n] \times I[n]}
\\ zS^\infty_n\calc  \ar[d]_{\text{forget}}^{\text{equiv}} \ar@{^{(}->}[u]^{\text{0-full}} \ar[r]_{G\vert_{zS^\infty_n\calc}} & \overline{zS^\infty_n}\calc \ar[d]^{\text{forget}}_{\text{equiv}} \ar@{^{(}->}[u]_{\text{0-full}}
\\ S_n^\infty \calc \ar[r] & \overline{S_n^\infty} \calc }$$
The middle map, which is the restriction of the fully faithful essentially surjective map $G$ to $zS^\infty_n\calc$, is also an equivalence: it is fully faithful because $G$ is and its domain and codomain are 0-full, and it is essentially surjective because $G(zS^\infty_n\calc)=\overline{zS^\infty_n}\calc$.

Finally, by the 3-for-2 property of equivalences, the bottom map $S_n^\infty \calc \to \overline{S_n^\infty} \calc$ is also an equivalence. This equivalence is exact and reflects cofibrations, so it is a Waldhausen equivalence by \cite[Proposition~6.9]{FiorePieper}.
\end{proof}

\subsection{The $\overline{\calf_\bullet^\infty}$ Construction} \label{subsec:F} \leavevmode \\

We next define the quasicategory $\overline{\calf_{n}^\infty}\calc$ of sequences of $n$-many cofibrations {\it without} selected quotients.
Let $\calc$ be a Waldhausen quasicategory. The quasicategory $\overline{\calf_{n}^\infty}\calc$ is the 0-full subsimplicial set of $\calc^{I[n]}$ on the functors $I[n] \to co \calc$. For a fibrational version of an $\overline{\calf_\bullet^\infty}$, see Barwick \cite[Section 5]{Barwick}.

As in the classical work of Waldhausen \cite[pages 324 and 328]{WaldhausenAlgKTheoryI}, the quasicategory $\overline{\calf_{n}^\infty}\calc$ is a Waldhausen quasicategory, the cofibrations in $\overline{\calf_{n}^\infty}\calc$ are level-wise cofibrations with the additional property that the map from each pushout to each lower corner is a cofibration in $\calc$.

\begin{proposition} \label{prop:SF_equivalence}
The forgetful functor $S^\infty_n\calc \to \overline{\calf_{n-1}^\infty}\calc$ is a Waldhausen equivalence of Waldhausen quasicategories.
\end{proposition}
\begin{proof}
This forgetful functor factors as the composite $S^\infty_n\calc \to \overline{S^\infty_n}\calc \to \overline{\calf_{n-1}^\infty}\calc$, the first map of which is a Waldhausen equivalence by Proposition~\ref{prop:Sn_equivalent_to_Snbar}. An inductive application of Proposition~\ref{prop:colimiting_cocone_equivalence}, forgetting one row at a time, shows that the second map is also an equivalence. Each of these intermediate functors is exact, reflects cofibrations, and is an equivalence, so is a Waldhausen equivalence by \cite[Proposition~6.9]{FiorePieper}. For instance, to apply Proposition~\ref{prop:colimiting_cocone_equivalence}, we consider row 1 and row 2 of a restricted $[n]$-complex as part of a colimiting cocone consisting of rows 1 and 2, a row of $A_1=A_1= \cdots =A_1$, and a row of $\ast= \ast = \cdots = \ast$, so that we have a row of pushout squares.
\end{proof}

Proposition~\ref{prop:SF_equivalence}, and the earlier results it references, provides some of the details to \cite[Lemma~7.3]{BlumbergGepnerTabuadaI} of Blumberg--Gepner--Tabuada.

The equivalence of quasicategories $S^\infty_n\calc \to \overline{\calf_{n-1}^\infty}\calc$ in Proposition~\ref{prop:SF_equivalence} induces an equivalence of homotopy categories since $ho$ is a 2-functor \cite[Proposition 1.27]{JoyalQuadern}.

\section{Approximation Theorems for Waldhausen Quasicategories} \label{sec:approximation_theorems}

We now turn to the proof of various Approximation Theorems for Waldhausen quasicategories, namely the Pre-Approximation Theorem~\ref{thm:Pre-Approximation}, the Approximation Theorem~\ref{thm:Approximation}, its Corollary~\ref{cor:Approx_hoF_equiv_and_coA=A}, the Cofibration Approximation Theorem~\ref{thm:CofibrationApproximation}, and its Corollaries~\ref{cor:Approximation_when_F_induces_cof_equiv} and \ref{cor:Approximation_when_F_rflcts_cofs_and_hoF_is_equivalence}. We also have a brief interlude in Section~\ref{subsec:App_Comparison} that compares three different approximation axioms: Pre-App 2, App 2, and Cofibration App 2.  .

The central result in this succession of Approximation Theorems is the Pre-Approximation Theorem~\ref{thm:Pre-Approximation}. Its proof uses the quasicategorical version of Quillen's Theorem A recalled in Section~\ref{subsec:Quasicategorical_Thm_A}, but does not require Waldhausen structures. The Pre-Approximation Theorem gives sufficient conditions for a functor to restrict to an equivalence of maximal $\infty$-groupoids.

\subsection{The Pre-Approximation Theorem} \leavevmode \\

We now turn to the main technical result in this paper, which has variants of Waldhausen Approximation among its consequences. The idea goes back to Waldhausen in the classical context of Approximation, namely use a version of Proposition~\ref{prop:extracted_from_Waldhausen} and Quillen's Theorem A. However, the implementation here is different. A key ingredient comes from Schlichting's Appendix \cite[page 132]{Schlichting}, where he proves Approximation for Waldhausen categories, assuming factorizations instead of a cylinder functor. There, Schlichting takes a colimit of a (cofibrant replacement\footnote{He uses his Lemma 13 to construct the cofibrant replacement diagram $Y$ in the proof on page 132.} of a) $\calp$-shaped diagram, and applies the 3-for-2 property to an appropriate commutative triangle. We proceed in this way (without cofibrant replacement) for a similar triangle in diagram \eqref{equ:To_Be_3-Simplex_in_B}.

We do not assume factorization in the domain quasicategory, nor do we assume a cylinder functor, nor even a Waldhausen structure for the Pre-Approximation Theorem~\ref{thm:Pre-Approximation}. For its notation, recall that if $\cala$ is a quasicategory, then $\cala_\text{equiv}$ is its maximal Kan subcomplex, which is the 1-full subquasicategory on the equivalences in $\cala$. Similarly, if $F$ is a functor of quasicategories (i.e. map of the underlying simplicial sets), then $F_\text{equiv}$ is the restriction of $F$ to the maximal Kan subcomplexes. For the proof of Theorem~\ref{thm:Pre-Approximation}, recall also the adjunction between
$\begin{array}{c}
\xymatrix@1{A \ar[d]  \\ A \star (-)}
\end{array}$ and $(A \to -) \backslash X$
in equation \eqref{equ:join_cocone_adjunction}, and the idea of a colimit as an initial object in the $a$-cocone quasicategory $a \backslash X$ in Section~\ref{subsec:colimits_in_a_quasicategory}. We will also make use of the adjunction between
$\begin{array}{c}
\xymatrix@1{B \ar[d]  \\ (-)\star B}
\end{array}$ and $X/(B \to -)$ in equation \eqref{equ:join_cone_adjunction}.

\begin{theorem}[Pre-Approximation] \label{thm:Pre-Approximation}
Let $F\co \cala \to \calb$ be a functor between quasicategories. No Waldhausen structures are assumed.  Suppose:
\begin{enumerate}
\item \label{thm:Pre-Approximation:reflects_equivs}
(Pre-App 1) \\
$F$ reflects equivalences.
\item \label{thm:Pre-Approximation:approx_axiom}
(Pre-App 2) \\
For every $a \in \cala$ and every morphism $Fa \to b$ in the codomain $\calb$, there exists a morphism $f \co a \to a'$ in $\cala$, an equivalence $F(a') \simeq b$ in $\calb$, and a 2-simplex in $\calb$ of the following form.
$$\xymatrix{Fa \ar[rr]^-{\forall} \ar[dr]_{\exists \;\; F(f)} & & b \\  & F(a') \ar[ur]_{\exists \;\;\text{\rm equivalence}} & }$$
\item \label{thm:Pre-Approximation:hoFequiv_esssurj_full}
The functor of groupoids $ho \big(F_\text{\rm equiv}\big)\co ho \big(\cala_\text{\rm equiv}\big) \to ho \big(\calb_\text{\rm equiv}\big)$ is essentially surjective and full. In other words,
$F$ is essentially surjective and $hoF$ is full on isomorphisms in the sense that if $[g]\co Fa \to Fa'$ is an isomorphism in $ho\calb$, then there exists an isomorphism $[f]\co a \to a'$ in $ho\cala$ such that $F[f]=[g]$.
\item \label{thm:Pre-Approximation:finite_posets_of_equivs_have_colimits}
The quasicategory $\cala$ admits colimits of diagrams in $\cala_\mathrm{equiv}$ indexed by connected finite posets, and $F$ preserves such colimits.
\\ (Note that $\cala_\mathrm{equiv}$ itself is not assumed to admit such colimits.)
\end{enumerate}
Then $F_\mathrm{equiv} \co \cala_\mathrm{equiv} \to \calb_\mathrm{equiv}$ is an equivalence of $\infty$-groupoids.\\

If $ho F\co ho\cala \to ho\calb$ is an equivalence of categories, then hypotheses \ref{thm:Pre-Approximation:reflects_equivs}, \ref{thm:Pre-Approximation:approx_axiom}, \ref{thm:Pre-Approximation:hoFequiv_esssurj_full} are satisfied. Thus, if $hoF$ is an equivalence of categories and \ref{thm:Pre-Approximation:finite_posets_of_equivs_have_colimits} holds, then $F_{\text{\rm equiv}}$ is an equivalence of $\infty$-groupoids.
\end{theorem}
\begin{proof}
Suppose the four hypotheses \ref{thm:Pre-Approximation:reflects_equivs}, \ref{thm:Pre-Approximation:approx_axiom}, \ref{thm:Pre-Approximation:hoFequiv_esssurj_full}, \ref{thm:Pre-Approximation:finite_posets_of_equivs_have_colimits} hold.  We show that $F_\text{\rm equiv}$ is a {\it weak} homotopy equivalence using the quasicategorical version of Quillen's Theorem A recalled above in Theorem~\ref{thm:Quasicategorical_Quillen_Theorem_A}. But since the domain and codomain of $F_\text{\rm equiv}$ are cofibrant-fibrant objects in the Kan model structure, this implies $F_\text{\rm equiv}$ is actually even a simplicial homotopy equivalence, and therefore an equivalence of quasicategories.

Let $y$ be a vertex in $\calb_\text{\rm equiv}$. Our goal is to show that $(F_\text{\rm equiv}\downarrow y)$ (see Definition~\ref{def:(f_downarrow_y)}) is weakly contractible using Proposition~\ref{prop:extracted_from_Waldhausen}. The over simplicial set $(F_\text{\rm equiv}\downarrow y)$ is non-empty and connected by hypothesis \ref{thm:Pre-Approximation:hoFequiv_esssurj_full}. Let $\calp$ be a connected finite poset and consider a map of simplicial sets $\Phi\co N\calp \to (F_\text{\rm equiv}\downarrow y)$. By the universal property of the pushout $(F_{\text{equiv}}\downarrow y)$, this corresponds to two maps $p\co N\calp \to (\calb_{\text{equiv}}\downarrow y)$ and $r\co N\calp \to \cala_{\text{equiv}}$ such that $q\circ p=F_\text{\rm equiv} \circ r$ and the two corresponding triangles commute in the left diagram of \eqref{NP_pullback_square}.

\begin{equation} \label{NP_pullback_square}
\begin{array}{c}
\xymatrix@C=3pc@R=3pc{
N\calp \ar@/_1pc/[ddr]_p \ar@/^1pc/[drr]^r
\ar[dr]^-\Phi \\
& (F_\text{\rm equiv} \downarrow y) \ar[d] \ar[r] \ar@{}[dr]|{\text{pullback}}
& \cala_\mathrm{equiv} \ar[d]^{F_\mathrm{equiv}} \\
& (\calb_\text{\rm equiv} \downarrow y) \ar[r]_-q & \calb_\text{\rm equiv} }
\end{array}
\hspace{-.4in}
\begin{array}{c}
\xymatrix@C=3pc@R=3pc{
(N\calp) \star t_1 \ar@/_1pc/@{.>}[ddr]_{\overline{p}} \ar@/^1pc/@{.>}[drr]^{\overline{r}} \\
& (F \downarrow y) \ar[d] \ar[r] \ar@{}[dr]|{}
& \cala \ar[d]^{F} \\
& (\calb \downarrow y) \ar[r]_-q & \calb }
\end{array}
\end{equation}

We show that $\Phi$ extends to $(N\calp)\star t_1$ by extending $p$ and $r$ to $\overline{p}$ and $\overline{r}$ on $(N\calp)\star t_1$ in such a way that $F\circ \overline{r}=q \circ \overline{p}$ as in the right diagram of \eqref{NP_pullback_square}, and then we observe that $\overline{r}$ actually has codomain $\cala_\text{equiv}$, and $\overline{p}$ actually has codomain $(\calb_\text{\rm equiv} \downarrow y)$, so that an extension $\overline{\Phi}\co (N\calp)\star t_1 \to (F_\text{\rm equiv} \downarrow y)$ is defined by the universal property of the pullback.

We use the notation $(N\calp)\star t_1$ where $t_1$ is the terminal simplicial set with sole vertex $t_1$. This is of course isomorphic to $(N\calp)\star 1$, but writing $t_1$ instead of 1 avoids ambiguities when we later use $(N\calp)\star \Delta[1]$. We will write $\Delta[1]$ as $\Delta[t_0,t_1]$, the nerve of the category with a single nontrivial morphism $t_0 \to t_1$. We shall also have occasion to use $\Delta[1] \cong \Delta[t_1, t_2]\cong t_1 \star t_2$ and $\Delta[1] \cong \Delta[t_0, t_2]\cong t_0 \star t_2$, and also $\Delta[2] \cong t_0 \star t_1 \star t_2$.

We extend $r$ to $\overline{r}$ in two steps: first extend $r$ to a universal cocone $\overline{r}'$, and then compose $\overline{r}'$ with a morphism $f$ in $\cala$ obtained from hypothesis \ref{thm:Pre-Approximation:approx_axiom} to form $\overline{r}$. For the first step, since $r$ takes values in $\cala_{\text{equiv}}$, we can form a universal cocone $\overline{r}'\co (N \calp) \star t_0 \to \cala$ by hypothesis \ref{thm:Pre-Approximation:finite_posets_of_equivs_have_colimits}. Its image $F \circ \overline{r}'$ is a universal cocone in $\calb$, also by \ref{thm:Pre-Approximation:finite_posets_of_equivs_have_colimits}.

For the second step of extending $r$ to $\overline{r}$, notice we have another cocone in $\calb$, and this one consists entirely of equivalences: the transpose $p^\dagger\co (N \calp) \star t_2 \to \calb_{\text{equiv}}$ of $p$ via adjunction \eqref{equ:join_cone_adjunction} is a (non-universal) cocone in $\calb_\text{equiv}$ with $p^\dagger(t_2)=y$. The transpose $p^\dagger$ is a cocone for $F_{\text{equiv}}\circ r$ because $p^\dagger\vert_{N\calp}=q \circ p$ by Remark~\ref{rem:remarks_on_over-quasicategory}, which is $F_{\text{equiv}}\circ r$. Since $F \circ \overline{r}'$ is an initial object of $(F\circ r) \backslash \calb$, there is a (homotopically unique) cocone morphism $\alpha\co N\calp \star \Delta[t_0,t_2] \to \calb$ such that
$$\xymatrix@R=3pc@C=3pc{N\calp \ar@{^{(}->}[d] \ar[dr]^{F \circ r} & \\ N\calp \star \Delta[t_0,t_2] \ar[r]_-\alpha & \calb }$$
commutes and
\begin{equation} \label{equ:conditions_on_alpha}
\alpha\vert\; N\calp \star t_0 = F \circ \overline{r}' \hspace{1in} \alpha\vert\; N\calp \star t_2 = p^\dagger.
\end{equation}
Next we apply hypothesis \ref{thm:Pre-Approximation:approx_axiom} to factor $\alpha\vert \;\emptyset \star \Delta[t_0,t_2]$ as
\begin{equation} \label{equ:Factorization_of_Alpha[1]}
\begin{array}{c}
\xymatrix{F\circ \overline{r}'(t_0) \ar[rr]^-{\alpha\vert \;\emptyset \star \Delta[t_0,t_2]} \ar[dr]_{F(f)} & & y \\  & F(a') \ar[ur]_{\text{\rm equivalence}} & }
\end{array}
\end{equation}
in $\calb$ using a morphism $f$ in $\cala$.

The cocone $\overline{r}'\co N\calp \star t_0 \to \cala$ in composes with the morphism $f\co \Delta[t_0,t_1] \to\cala$ to give a new cocone $\overline{r}\co N \calp \star t_1 \to \cala$. Namely, the left vertical map below is mid anodyne by \cite[Theorem~3.17~(i)]{JoyalQuadern}
\begin{equation} \label{equ:composing_a_cocone_with_a_morphism}
\begin{array}{c}
\xymatrix@C=5pc@R=3pc{(N\calp\star t_0)\cup(\emptyset \star \Delta[t_0,t_1]) \ar[r]^-{\overline{r}'\cup f} \ar@{^{(}->}[d]_{\text{mid anodyne}} & \cala \ar[d]^{\text{mid fibration}} \\
N\calp \star \Delta[t_0,t_1] \ar[r] \ar@{.>}[ru] & \ast }
\end{array}
\end{equation}
because $\Lambda^0[1]=\{0\} \to \Delta[1]$ is left anodyne. Here we write $\Delta[t_0,t_1]\cong \Delta[1]$ to signify the nerve of the category $t_0 \to t_1$, as mentioned earlier. The restriction of the diagonal lift to $N\calp \star t_1$ is $\overline{r}:=f\circ \overline{r}'$. This completes the extension of $r$ to the cocone $\overline{r}\co N \calp \star t_1 \to \cala$.

We claim that $\overline{r}\co N\calp \star t_1 \to \cala$ sends every morphism to an equivalence. Let $s$ be an object of $N\calp$ and consider the unique morphisms $u$, $v$, $w$ from $s$ to $t_0$, $t_1$, $t_2$ in $N \calp \star (t_0 \star t_1\star t_2)$, as in the following diagram of three 2-simplices.
\begin{equation} \label{equ:Triangles_in_NP*t123}
\begin{array}{c}
\xymatrix@C=4pc{s \ar[r]_u \ar@/^1pc/[rr]^v \ar[dr]_w & t_0 \ar[r] \ar[d] & t_1 \ar[dl] \\ & t_2 & }
\end{array}
\hspace{.5in} \text{in \;\; $N \calp \star (t_0 \star t_1\star t_2)$}
\end{equation}
We construct a diagram in $\calb$ of 2-simplices in \eqref{equ:To_Be_3-Simplex_in_B} by mapping parts of \eqref{equ:Triangles_in_NP*t123} individually. The composite of $F$ with the filler in \eqref{equ:composing_a_cocone_with_a_morphism} maps the top 2-simplex of \eqref{equ:Triangles_in_NP*t123}
to the top 2-simplex of \eqref{equ:To_Be_3-Simplex_in_B}.
\begin{equation} \label{equ:To_Be_3-Simplex_in_B}
\begin{array}{c}
\xymatrix@C=6.5pc@R=5.5pc{F\circ\overline{r}'(s) \ar[r]_-{F \circ \overline{r}'(u)} \ar@/^2pc/[rr]^{F\circ \overline{r}(v)} \ar[dr]_{\text{equivalence}}^{p^\dagger(w)} & F \circ \overline{r}'(t_0) \ar[r]_{F(f)} \ar[d]^{\alpha\vert \; \emptyset \star \Delta[t_0,t_2]} & F(a') \ar[dl]^{\text{equivalence}} \\ & y & }
\end{array} \hspace{.5in} \text{in \;\; $\calb$}
\end{equation}
The lower left 2-simplex is the $\alpha$-image of the lower left triangle of \eqref{equ:Triangles_in_NP*t123} because of the equalities \eqref{equ:conditions_on_alpha}. The lower right 2-simplex of \eqref{equ:To_Be_3-Simplex_in_B} is the 2-simplex \eqref{equ:Factorization_of_Alpha[1]}. All together, diagram \eqref{equ:To_Be_3-Simplex_in_B} is a $\Lambda^1[3]$-horn in $\calb$, so it extends to a 3-simplex, and in particular, we obtain a 2-simplex with boundary the outermost triangle. Thus, by the 3-for-2 property of equivalences, the top morphism $F\circ \overline{r}(v)$ is an equivalence. By hypothesis \ref{thm:Pre-Approximation:reflects_equivs}, we have that $\overline{r}(v)$ is also an equivalence. We already know that $\overline{r}$ maps all morphisms of $N\calp$ to equivalences, as $\overline{r}\vert_{N\calp}=r$, so $\overline{r}\co N\calp \star t_1 \to \cala$ sends every morphism to an equivalence as claimed.

Next we construct the extension $\overline{p}\co N\calp \star t_1 \to (\calb \downarrow y)$ from our already completed work.
The inclusion $\Lambda^0[2] \hookrightarrow \Delta[2]$ is left anodyne, so
the left vertical map below is mid anodyne by \cite[Theorem~3.17~(iii)]{JoyalQuadern}.
\begin{equation} \label{equ:towards_pbar1}
\begin{array}{c}
\xymatrix@C=5pc@R=3pc{\big(N\calp \star \Lambda^0[2]\big) \; \bigcup \; \big(\emptyset \star \Delta[2]\big) \ar[r]^-{\eqref{equ:towards_pbar2}} \ar@{^{(}->}[d]_{\text{mid anodyne}} & \calb \ar[d]^{\text{mid fibration}} \\
N\calp \star (t_0 \star t_1 \star t_2) \ar[r] \ar@{.>}[ru] & \ast }
\end{array}
\end{equation}
The top horizontal morphism in \eqref{equ:towards_pbar1} is the following map, which is like \eqref{equ:To_Be_3-Simplex_in_B} with $s$ replaced by $N\calp$.
\begin{equation} \label{equ:towards_pbar2}
\begin{array}{c}
\xymatrix{\big(N\calp \star \Lambda^0[2]\big) \; \bigcup \; \big(\emptyset \star \Delta[2]\big) \ar@{=}[d] \\
\big( N\calp \star (t_0\star t_1 \cup t_0 \star t_2)\big)\; \bigcup \; (t_0\star t_1 \star t_2) \hspace{.65in} \ar@{=}[d] \\
 N\calp \star (t_0\star t_1) \; \bigcup \; N\calp \star (t_0\star t_2) \; \bigcup \; (t_0\star t_1 \star t_2) \hspace{1.1in} \ar[d]_-{\eqref{equ:composing_a_cocone_with_a_morphism} \; \bigcup \; \alpha \; \bigcup \; \eqref{equ:Factorization_of_Alpha[1]}} \\ \calb}
\end{array}
\end{equation}
The second and third objects are equal by the compatibility of $N\calp\star -$ with union by \cite[Lemma~3.8]{JoyalQuadern}. A diagonal lift in \eqref{equ:towards_pbar1} exists, and we denote its restriction to $N\calp \star (t_1 \star t_2)$ by $$\xymatrix{\overline{p}^\dagger\co N\calp \star (t_1 \star t_2) \ar[r] & \calb}.$$
The restriction is the analogue of the outermost triangle (i.e. the composite triangle) of \eqref{equ:To_Be_3-Simplex_in_B}. In fact, $\overline{p}^\dagger$ goes into $\calb_\text{equiv}$ by the comments after \eqref{equ:To_Be_3-Simplex_in_B}. The transpose of $\overline{p}^\dagger$ is the desired map
$$\xymatrix{\overline{p}\co N\calp \star t_1 \ar[r] & (y \downarrow \calb_\text{equiv})}.$$ Hence we have constructed the extension $\overline{\Phi}\co (N\calp)\star t_1 \to (F_\text{\rm equiv} \downarrow y)$ by the universal property of the pullback, so $(F_\text{\rm equiv} \downarrow y)$ is contractible by Proposition~\ref{prop:extracted_from_Waldhausen}, and $F_\text{\rm equiv}$ is a weak homotopy equivalence by the quasicategorical version of Quillen's Theorem A, recalled above in Theorem~\ref{thm:Quasicategorical_Quillen_Theorem_A}.

Suppose now that $ho F\co ho\cala \to ho\calb$ is an equivalence of categories. We show that hypotheses \ref{thm:Pre-Approximation:reflects_equivs}, \ref{thm:Pre-Approximation:approx_axiom},  \ref{thm:Pre-Approximation:hoFequiv_esssurj_full} are satisfied. For \ref{thm:Pre-Approximation:reflects_equivs}, $F$ reflects equivalences because $ho F$ reflects isomorphisms. To construct the factorization in \ref{thm:Pre-Approximation:approx_axiom} of given $h \co Fa \to b$, we use essential surjectivity of $ho F$ to find an equivalence $g\co Fa' \to b$, then we use the fully faithfulness of $ho F$ to find a homotopy class $[f]$ in $\cala$ with
$$[F(f)]=[g]^{-1}\circ [h].$$ Finally, $[g]\circ[F(f)]=[h]$, so a 2-simplex as in \ref{thm:Pre-Approximation:approx_axiom} exists.  For \ref{thm:Pre-Approximation:hoFequiv_esssurj_full}, $ho\big(F_\text{equiv}\big)$ is essentially surjective because $ho(F)$ is, and $ho\big(F_\text{equiv}\big)$ is full because $ho(F)$ is full and reflects equivalences ($ho(F)$ is even fully faithful).
\end{proof}

\begin{remark} \label{rem:Pre-App2+initialobj_implies_essentially_surjective}
Pre-App 2, stated as hypothesis~\ref{thm:Pre-Approximation:approx_axiom} of Theorem~\ref{thm:Pre-Approximation}, implies that $F$ is essentially surjective when there is an initial object in the image of $F$. This is the case for an exact functor between Waldhausen quasicategories.
\end{remark}

\begin{remark} \label{rem:illustration_of_colimit_hypothesis}
The following example illustrates that the colimit condition stated in hypothesis~\ref{thm:Pre-Approximation:finite_posets_of_equivs_have_colimits} of Theorem~\ref{thm:Pre-Approximation} is not an empty condition. Consider the connected finite poset
$$\cali=\left\{ \begin{array}{c} \xymatrix{\bullet \ar[r] \ar[d] & \bullet  & \bullet \ar[l] \ar[d] \\ \bullet & & \bullet \\ \bullet \ar[r] \ar[u] & \bullet  & \bullet \ar[l] \ar[u] } \end{array} \right\}$$
and the constant diagram $\ast$ indexed by $\cali$ in the quasicategory of Kan complexes. This diagram is a diagram of equivalences, in fact a diagram of equalities of the selected terminal Kan complex $\ast$.
Then the {\it colimit} of this constant diagram in this quasicategory is the {\it homotopy colimit}
$\text{hocolim}_\cali\, \ast$ in the Kan enriched category of Kan complexes \cite[Theorem 4.2.4.1]{LurieHigherToposTheory}. By Bousfield-Kan  \cite{BousfieldKanBook}, $\text{hocolim}_\cali\, \ast$ has the homotopy type of $N\cali\simeq S^1 \not\simeq \ast$. Thus, the colimit condition \ref{thm:Pre-Approximation:finite_posets_of_equivs_have_colimits} on a diagram of equivalences in a {\it quasicategory} is genuinely different from the analogous condition on a diagram of isomorphisms in a {\it category}. In a category, a diagram of isomorphisms indexed by a connected finite poset always has a colimit: simply select any object in the diagram, and take as a colimiting cocone the paths to that object.
\end{remark}

\subsection{The Approximation Theorem} \leavevmode \\

In this section we prove the Approximation Theorem~\ref{thm:Approximation} using the Pre-Approximation Theorem~\ref{thm:Pre-Approximation} and Lemma~\ref{lem:ImageOf_hoF}. A comparison with other recent Approximation Theorems in the literature is in the subsequent Section~\ref{subsec:ComparisonWithOtherAuthors}.

A familiar tool in work on Approximation is to consider the coproduct $x \vee y$ and make a factorization as in \eqref{equ:factorization_of_sum}, and I thank Georgios Raptis for suggesting the following Lemma. Recall that if $x$ and $y$ are objects in a Waldhausen quasicategory $\cala$, then we denote any pushout of $x \leftarrow 0 \rightarrow y$ by $x \vee y$, and indicate any morphism induced by the universal property as $f_1 + f_2\co x \vee y \to z$. In particular, if $F$ is an exact functor, then $F(x \vee y)$ is a $F(x) \vee F(y)$ and we can speak of morphism sums out of $F(x \vee y)$, without making any identifications.

\begin{lemma}[Condition for a Homotopy Class to be in Image of $hoF$ or $ho F_{\text{\rm equiv}}$] \label{lem:ImageOf_hoF}
Let $F\co \cala \to \calb$ be an exact functor of Waldhausen quasicategories and suppose $F$ reflects equivalences. Let $x$ and $y$ be objects in $\cala$.  Suppose $g\co Fx \to Fy $ is a morphism in $\calb$ and suppose the map
$$\xymatrix{g + F(\text{\rm Id}_y) \co F(x\vee y) \ar[r] & F(y)}$$
factors as $Ff$ followed by an equivalence $w$.
\begin{equation} \label{equ:factorization_of_sum}
\xymatrix{F(x\vee y) \ar[r]^-{Ff} \ar@/_1.7pc/[rr]_{g + F(\text{\rm Id}_y)} & F(a') \ar[r]^-w_\simeq & F(y)}
\end{equation}
\begin{enumerate}
\item
Then $[g]$ is in the image of $ho(F)$. More precisely, there exists a homotopy class $[\overline{g}]\co x \to y$ in $ho\cala$ such that $\big(hoF([\overline{g}])\big)=[g]$.
\item
Moreover, if $g$ is an equivalence, then $[g]$ is in the image of $ho\big(F_{\text{\rm equiv}}\big)$.
\end{enumerate}
\end{lemma}
\begin{proof} \leavevmode
\begin{enumerate}
\item
The two triangles
$$\xymatrix{ & F(x) \ar[dl]_{F(i_x)} \ar[dr]^g & \\
F(x\vee y) \ar[rr]^-{g + F(\text{\rm Id}_y)} & & F(y) \\
& F(y) \ar[ul]^{F(i_y)} \ar[ur]_{F(\text{Id}_y)=\text{Id}_{Fy}} & }$$
commute, so we factor the middle map according to \eqref{equ:factorization_of_sum}, and form the dotted composites below, to have 4 commutative triangles.
$$\xymatrix@C=3pc@R=3pc{ & F(x) \ar[dl]_{F(i_x)} \ar[dr]^g \ar@{.>}[d]|{F(f i_x)} & \\
F(x\vee y) \ar[r]^-{Ff} & F(a') \ar[r]^-w_\simeq & F(y) \\
& F(y) \ar[ul]^{F(i_y)} \ar[ur]_{F(\text{Id}_y)=\text{Id}_{Fy}} \ar@{.>}[u]|{F(f i_y)} & }$$
By 3-for-2 in the bottom right triangle, $F(f i_y)$ is an equivalence, so $f i_y$ is an equivalence, as $F$ reflects equivalences. From the two right triangles we now have $F((f i_y)^{-1} \circ (f i_x))$ is homotopic to $g$, so we may take $\overline{g}$ to be any composite $(f i_y)^{-1} \circ (f i_x)$.
\item
If $g$ is an equivalence, then $\overline{g}$ is also an equivalence, as $F$ reflects equivalences, so $[g]$ is in the image of $ho F_{\text{\rm equiv}}$.
\end{enumerate}
\end{proof}

\begin{proposition}
If $F$ satisfies App 1 and App 2, then both $hoF$ and $ho \big( F_{\text{\rm equiv}}\big)$ are full.
\end{proposition}
\begin{proof}
This follows directly from Lemma~\ref{lem:ImageOf_hoF}.
\end{proof}

\begin{theorem}[Waldhausen Approximation for Quasicategories] \label{thm:Approximation}
Let $F\co \cala \to \calb$ be an exact functor between Waldhausen quasicategories. Suppose:
\begin{enumerate}
\item \label{thm:Approximation:App1}
(App 1) \\
$F$ reflects equivalences.
\item \label{thm:Approximation:App2}
(App 2) \\
For every $a \in \cala$ and every morphism $Fa \to b$ in the codomain $\calb$, there exists a cofibration $f \co a \rightarrowtail a'$ in $\cala$, an equivalence $F(a') \simeq b$ in $\calb$, and a 2-simplex in $\calb$ of the following form.
$$\xymatrix{Fa \ar[rr]^-{\forall} \ar@{>->}[dr]_{\exists \;\; F(f)} & & b \\  & F(a') \ar[ur]_{\exists \;\;\text{\rm equivalence}} & }$$
\item \label{thm:Approximation:finite_colimits}
The domain quasicategory $\cala$ admits all finite colimits, and $F$ preserves them (both are true for instance if all morphisms of the domain quasicategory $\cala$ are cofibrations).
\end{enumerate}
Then $F_{\text{\rm equiv}}$ is an equivalence of $\infty$-groupoids, as is $\big(S_n^\infty F\big)_{\text{\rm equiv}}$ for every $n \geq 0$. Consequently $\bfK(F)$ is a level-wise equivalence of spectra.
\end{theorem}
\begin{proof}
We combine the Pre-Approximation Theorem~\ref{thm:Pre-Approximation} with Lemma~\ref{lem:ImageOf_hoF}. For \ref{thm:Pre-Approximation:reflects_equivs} and \ref{thm:Pre-Approximation:approx_axiom} of the Pre-Approximation Theorem, Pre-App 1 and Pre-App 2 hold by assumption, as Pre-App 1 = App 1, and App 2 implies Pre-App 2. For \ref{thm:Pre-Approximation:hoFequiv_esssurj_full}, the functor $ho\big(F_\text{equiv}\big)$ is essentially surjective by Remark~\ref{rem:Pre-App2+initialobj_implies_essentially_surjective}, and is full by Lemma~\ref{lem:ImageOf_hoF}.
The special finite colimits required in \ref{thm:Pre-Approximation:finite_posets_of_equivs_have_colimits} of the Pre-Approximation Theorem exist in $\cala$ and are preserved by $F$ by hypothesis.


The same argument applies to $S_n^\infty F$ to conclude $\big(S_n^\infty F\big)_{\text{\rm equiv}}$ is an equivalence of $\infty$-groupoids, since $S_nF$ satisfies App 1 and App 2 when $F$ does by \cite[Lemma~1.6.6]{WaldhausenAlgKTheoryI}, see Proposition~\ref{prop:Cof_Apps_F_Imply_Cof_Apps_SnF} for the translation to quasicategories. Also, the quasicategory $S_n\cala$ admits finite colimits when $\cala$ does, and $S_nF$ preserves finite colimits when $F$ does.

Since $\big(S_n^\infty F \big)_\text{equiv}$ is a weak homotopy equivalence of simplicial sets for all $n \geq0$, the Realization Lemma for bisimplicial sets implies $\bfK_1(F)$ is weak homotopy equivalence of spaces.

The $K$-theory spaces recalled in Definition \ref{def:nth-K-theory-space} form an $\Omega$-spectrum beyond the 0-th term. Since $\bfK_1(F)$ is a weak homotopy equivalence of spaces, the map $\Omega \bfK_1(F)$ is also a weak homotopy equivalence of spaces (recall $\pi_n \Omega X \cong \pi_{n+1} X$ naturally), so 3-for-2 and the commutative diagram
\begin{equation} \label{equ:level-wise_equivalence_of_K-theory_spectra}
\begin{array}{c}
\xymatrix@R=3pc@C=3pc{\Omega \bfK_1(\cala) \ar[r]^{\mathrm{w.h.e.}} \ar[d]_{\Omega \bfK_1(F)}^{\mathrm{w.h.e.}} & \bfK_2(\cala) \ar[d]^{\bfK(F)_2} \\ \Omega \bfK_1(\calb) \ar[r]_{\mathrm{w.h.e.}} & \bfK_2(\calb) }
\end{array}
\end{equation}
imply that $\bfK_2(F)$ is a weak homotopy equivalence. Continuing in this way, every $\bfK_n(F)$ is a weak homotopy equivalence for $n \geq 1$.
\end{proof}

\begin{remark}
In the Approximation Theorem~\ref{thm:Approximation}, hypothesis \ref{thm:Approximation:finite_colimits} could be weakened to the existence and preservation of the special colimits in hypothesis \ref{thm:Pre-Approximation:finite_posets_of_equivs_have_colimits} of the Pre-Approximation Theorem. In the Approximation Theorem we require finite colimits rather than merely the special ones because the statement is easier to formulate and check in practice.
\end{remark}

\begin{remark} \label{rem:App2_implies_all_cofs}
Waldhausen's original condition App 2, namely that every morphism $F(a) \to b$ factors as $(\text{equiv}) \circ F(\text{cof})$, implies {\it in the quasicategorical context} that every morphism in the codomain is a cofibration. To see this, suppose $g\co b \to b'$ is any morphism in $\calb$. By essential surjectivity, there is an equivalence $F(a) \simeq b$, and by App 2 every morphism $F(a) \to b'$ is a cofibration, so any composite
$$\xymatrix{Fa \ar[r]^-{\simeq} \ar@/_1pc/@{>->}[rr] & b \ar[r]^g & b'}$$

\noindent is a cofibration.
Passing to the homotopy category and inverting the resulting isomorphism expresses $[g]=[\text{cof}] \circ (\text{iso})$, so $g$ is a cofibration. Since Waldhausen's App 2 implies every codomain morphism is a cofibration in the quasicategorical setting, it is natural to search for an Approximation Theorem with a weaker version of App 2: the factorization {\it only of cofibrations} $F(a) \rightarrowtail b$ as $(\text{equiv}) \circ F(\text{cof})$. This is the topic of Sections~\ref{subsec:App_Comparison} and \ref{subsec:TheCofibrationApproximationTheorem}.
\end{remark}

\begin{corollary}[Approximation when $hoF$ is an Equivalence of Homotopy Categories and $co\cala=\cala$] \label{cor:Approx_hoF_equiv_and_coA=A}
Let $F\co \cala \to \calb$ be an exact functor between Waldhausen quasicategories. Suppose $F$ induces an equivalence of homotopy categories, and suppose every morphism of $\cala$ is a cofibration. Then $F_{\text{\rm equiv}}$ is an equivalence of $\infty$-groupoids, as is $\big(S_n^\infty F\big)_{\text{\rm equiv}}$ for every $n \geq 0$. Consequently $\bfK(F)$ is a level-wise equivalence of spectra.
\end{corollary}
\begin{proof}
Since $F$ induces an equivalence of homotopy categories, $F$ reflects equivalences, so satisfies App 1. For App 2, suppose $g \co Fa \to b$ is a morphism in $\calb$. Then by the essential surjectivity of $hoF$, we have an equivalence $w\co F(a') \to b$.  The fullness of $hoF$ provides a pre-image homotopy class for $[w]^{-1} \circ [g]$, and any representative completes the required triangle in App 2 (recall every morphism of $\cala$ is a cofibration). Since every morphism of $\cala$ is a cofibration, $\cala$ admits pushouts and initial objects, so admits all finite colimits. The functor $F$ preserves pushouts and initial objects, so also preserves all finite colimits. Theorem~\ref{thm:Approximation} now applies.
\end{proof}

\begin{remark} \label{rem:exact_functor-equiv_on_hocats-domain_all_cofibs|implies|codom_all_cofibs}
In the situation of Corollary~\ref{cor:Approx_hoF_equiv_and_coA=A}, every morphism $g$ of the codomain is also a cofibration, since the equivalence of homotopy categories guarantees for each $g$ in the codomain the existence of some cofibration $f$ in the domain and some equivalences $v$ and $w$ in the codomain such that $[g]=[w][Ff][v]$.

Also notice, in Corollary~\ref{cor:Approx_hoF_equiv_and_coA=A}, the assumption that $co\cala=\cala$ is equivalent to the assumption that $\cala$ admits factorization, see Section~\ref{subsec:Review_of_WaldhausenQCats}.
\end{remark}

\begin{remark}
In Corollary~\ref{cor:Approx_hoF_equiv_and_coA=A}, if we drop the hypothesis $co\cala=\cala$ we can still conclude $F_\text{equiv}$ is an equivalence of $\infty$-groupoids by the Pre-Approximation Theorem~\ref{thm:Pre-Approximation}, but we cannot conclude anything about $\big(S_n^\infty F\big)_{\text{\rm equiv}}$. Properties Pre-App 1 and 2 for $F$ do not persist to $S_n^\infty F$ in general.
\end{remark}

\subsection{Comparison with Approximation Theorems in Classical Setting, Simplicially Enriched Setting, and Quasicategorical Setting} \label{subsec:ComparisonWithOtherAuthors} \leavevmode \\

There are several related results in the recent literature. First we recall the results in the {\it classical categorical setting}, some of which relate App 1 and App 2 to an equivalence of homotopy categories and simplicially enriched categories. For an ordinary 1-category $\cala$ with weak equivalences, the {\it homotopy category } $Ho(\cala)$ in these statements means the category obtained by formally inverting the weak equivalences. Some of these statements are slight reformulations of the original statements.

\bigskip
\noindent {\bf Theorem 1.6.7 of Waldhausen in \cite{WaldhausenAlgKTheoryI}}. {\it Let $\cala$ and $\calb$ be Waldhausen categories. Suppose the weak equivalences in both $\cala$ and $\calb$ satisfy the 3-for-2 property. Suppose further that $\cala$ has a cylinder functor and the weak equivalences in $\cala$ satisfy the cylinder axiom. Let $F\colon \cala \to \calb$ be an exact functor. Suppose $F$ satisfies App 1 and App 2. Then the induced maps $wF\colon w\cala \to w\calb$ and $wS_\bullet F \colon wS_\bullet \cala \to wS_\bullet \calb$ are weak homotopy equivalences.\footnote{In this cited Theorem 1.6.7 of \cite{WaldhausenAlgKTheoryI}, Waldhausen writes ``homotopy equivalence'' instead of ``weak homotopy equivalence'' to indicate that the geometric realizations of $wF$ and $wS_\bullet F$ are homotopy equivalences between CW-complexes.} }
\bigskip

\bigskip
\noindent {\bf Theorem 10 of Schlichting in \cite{Schlichting}}. {\it Let $\cala$ and $\calb$ be classical
Waldhausen categories. Suppose the weak equivalences in $\cala$ and $\calb$ both have the 3-for-2 property. Suppose every morphism in $\cala$ factors as a cofibration followed by a weak equivalence. Let $F \co \cala \to \calb$ be an exact functor
satisfying App 1 and App 2. Then the induced maps $wF\colon w\cala \to w\calb$ and $wS_\bullet F \colon wS_\bullet \cala \to wS_\bullet \calb$ are weak homotopy equivalences.}
\bigskip

\bigskip
\noindent {\bf Theorem 2.9 of Cisinski in \cite{CisinskiInvarianceOfK-Theory}}. {\it Let $\cala$ and $\calb$ be Waldhausen categories that are ``derivable'' in the sense that their weak equivalences satisfy the 3-for-2 property and every one of their morphisms can be factored as a cofibration followed by a weak equivalence. Suppose also that $\cala$ and $\calb$ are each ``strongly saturated'' in the sense that any morphism is a weak equivalence if and only if its image in $Ho(\cala)$ respectively $Ho(\calb)$ is an isomorphism. Let $F\co \cala \to \calb$ be a right exact functor. Then the following conditions are equivalent.
\begin{enumerate}
\item
For every finite ordered set $E$, the induced functor $Ho(\cala^E) \to Ho(\calb^E)$ is an equivalence of categories.
\item
The induced functor $Ho(\cala) \to Ho(\calb)$ is an equivalence of categories.
\item
The functor $F$ satisfies App 1 and App 2.
\end{enumerate} }
\bigskip

\bigskip
\noindent {\bf Theorem 3.25 of Cisinski in \cite{CisinskiInvarianceOfK-Theory}}. {\it Let $F\co \cala \to \calb$ be a left exact functor between categories of fibrant objects. If the induced functor $Ho(\cala) \to Ho(\calb)$ is an equivalence of categories, then $F$ induces an equivalence of simplicially enriched categories
$$\xymatrix{L^H(\cala) \ar[r] & L^H(\calb).}$$
Here $L^H$ is the hammock localization of Dywer and Kan.}
\bigskip

\bigskip
\noindent {\bf Theorem 1.5 of Blumberg-Mandell in \cite{BlumbergMandellAbstHomTheory}}. {\it Let $\cala$ be a classical Waldhausen category in which the weak equivalences have the 3-for-2 property and in which every morphism factors as a cofibration followed by a weak equivalence. Let $\calb$ be a classical Waldhausen category in which the weak equivalences have the 3-for-2 property. Let $F \co \cala \to \calb$ be an exact functor. If $F$ satisfies App 1 and App 2, then $Ho(\cala) \to Ho(\calb)$ is an equivalence of categories, and $Ho(a/\cala) \to Ho(Fa/\calb)$ is an equivalence of categories for all objects $a$ of $\cala$. Here $a/\cala$ and $Fa/\calb$ mean the categories under $a$ and $Fa$. }
\bigskip

\bigskip
\noindent {\bf Theorem 1.4 of Blumberg-Mandell in \cite{BlumbergMandellAbstHomTheory}}. {\it Let $\cala$ and $\calb$ be classical Waldhausen categories, both of which have the 3-for-2 property for their weak equivalences. Suppose further that every morphism in both categories factors as a cofibration followed by a weak equivalence.  Let $F\co \cala \to \calb$ be a functor that preserves weak equivalences. Then $F$ induces a Dwyer-Kan equivalence if and
only if $F$ induces an equivalence $Ho(\cala)\to Ho(\calb)$ and an equivalence $Ho(a/\cala) \to Ho(Fa/\calb)$ for all objects $a$ of $\cala$.}
\bigskip

Theorems 1.5 and 1.4 of Blumberg-Mandell together imply: if $F$ is an exact functor between classical Waldhausen categories that both satisfy 3-for-2 and factorization, and $F$ satisfies App 1 and App 2, then $F$ induces a Dwyer-Kan equivalence of hammock localizations.

\bigskip
\noindent {\bf Theorem 1.3 of Blumberg-Mandell in \cite{BlumbergMandellAbstHomTheory}}. {\it Let $\cala$ and $\calb$ be classical Waldhausen categories, both of which have the 3-for-2 property for their weak equivalences. Suppose further that every morphism in both categories factors as a cofibration followed by a weak equivalence.  Let $F\co \cala \to \calb$ be a a weakly exact functor that induces an equivalence on homotopy categories. If
any one of the following additional hypotheses holds
\begin{enumerate}
\item
the weak equivalences of $\cala$ and $\calb$ are closed under retracts,
\item
a morphism $f$ in $\cala$ is a weak equivalence if and only the morphism $Ff$ in $\calb$ is a weak equivalence, or
\item
for any objects $a,a$ in $\cala$, the image of $Ho(w\cala)(a,a')$ in $Ho\calb(Fa,Fa')$ coincides with the image
of $Ho(w\calb)(Fa,Fa')$,
\end{enumerate}
then $F$ induces an equivalence of $K$-theory spectra.}
\bigskip

Barwick has proved Corollary~\ref{cor:Approx_hoF_equiv_and_coA=A}, but not Theorem~\ref{thm:Approximation}. His Corollary~8.2.2 of \cite{Barwick} implies: if $G\co \cala \to \calb$ is an exact functor between Waldhausen quasicategories, {\it both of which have all maps cofibrations}, and $G$ induces an equivalence $ho\cala \to ho\calb$ of homotopy categories, then $G$ induces a stable equivalence of $K$-theory spectra. Another version of Corollary~\ref{cor:Approx_hoF_equiv_and_coA=A} was proved by Blumberg--Gepner--Tabuada for {\it stable} quasicategories: their  Corollary~5.11 of \cite{BlumbergGepnerTabuadaI} states that a map of stable quasicategories is an equivalence if and only if it induces an equivalence of homotopy categories. See present Corollaries~\ref{cor:Approximation_when_F_induces_cof_equiv} and \ref{cor:Approximation_when_F_rflcts_cofs_and_hoF_is_equivalence} for further quasicategorical settings in which an equivalence of cofibration homotopy categories or an equivalence of homotopy categories guarantees a levelwise equivalence of $K$-theory spectra.

In light of Remark~\ref{rem:exact_functor-equiv_on_hocats-domain_all_cofibs|implies|codom_all_cofibs}, we can point out two predecessors to Corollary~\ref{cor:Approx_hoF_equiv_and_coA=A} in the setting of classical Waldhausen categories. Blumberg-Mandell prove in \cite[Theorem 1.3 (ii)]{BlumbergMandellAbstHomTheory}: if $\calc$ and $\cald$ are both classical Waldhausen categories which admit factorization and whose classes of weak equivalences both have the 3-for-2 property, then any {\it weakly} exact functor $G \co \calc \to \cald$ that reflects weak equivalences and induces an equivalence of homotopy categories induces a stable equivalence of $K$-theory spectra. Cisinski also made this same conclusion in \cite{CisinskiInvarianceOfK-Theory} for a right exact functor $G \co \calc \to \cald$ for which each Waldhausen category $\calc$ and $\cald$ is a category of cofibrant objects, has a null object, and satisfies saturation conditions.

\subsection{The Cofibration Approximation Axiom 2 and its Comparison with App 2 and Pre-App 2} \label{subsec:App_Comparison} \leavevmode \\

Since Waldhausen's condition App 2 (in the quasicategorical setting) implies every morphism of the codomain is a cofibration  by Remark~\ref{rem:App2_implies_all_cofs}, we would like a more general version the Approximation Theorem~\ref{thm:Approximation} without every codomain morphism a cofibration. So, we replace the hypothesis App 2 by a weaker condition called {\it Cofibration App 2} that requires only factorization of cofibrations $F(a) \rightarrowtail b$ and prove the Cofibration Approximation Theorem~\ref{thm:CofibrationApproximation} in the next section. We now introduce Cofibration App 2, compare it with App 2 and Pre-App 2 in Remark~\ref{rem:AppsRelationships}, consider invariance of all three properties under composition in Proposition~\ref{prop:compositions_satisfy_App2}~\ref{prop:compositions_satisfy_App2:hyp:FE} and \ref{prop:compositions_satisfy_App2:hyp:ho(coG)}, invariance of Cofibration App 2 under natural equivalence in Proposition~\ref{prop:compositions_satisfy_App2}~\ref{prop:compositions_satisfy_App2:hyp:natural_equivalence}, the persistence  of Cofibration App 2 to $S_n^\infty F$ and $\overline{\mathcal{F}_{n-1}^\infty}(F)$ in Proposition~\ref{prop:Cof_Apps_F_Imply_Cof_Apps_SnF}, and the persistence of Pre-App 2 in Proposition~\ref{prop:Pre-Apps_for_F_and_coA=A_imply_Pre-Apps_For_SnF}.

\begin{definition} \label{def:Cof_Apps}
An exact functor $F\co \cala \to \calb$ between Waldhausen quasicategories has the {\it Cofibration Approximation Property} if its restriction $coF\co co\cala \to co\calb$ to the cofibration subquasicategories satisfies Pre-App 1 and Pre-App 2 of Theorem~\ref{thm:Pre-Approximation}. More specifically the following two conditions on $F$ are required to hold. \\ \\
{\it Cofibration App 1. } The exact functor $F$ reflects equivalences.  \\ \\
{\it Cofibration App 2. } For every $a \in \cala$ and every {\it cofibration} $g\co Fa \rightarrowtail b$ in the codomain $\calb$, there exists a {\it cofibration} $f \co a \rightarrowtail a'$ in $\cala$, an equivalence $F(a') \simeq b$ in $\calb$, and a 2-simplex in $\calb$ of the following form.
$$\xymatrix{Fa \ar@{>->}[rr]^-{\forall \text{ cofibration} \; g} \ar@{>->}[dr]_{\exists \;\; F(f)} & & b \\  & F(a') \ar[ur]_{\exists \;\;\text{\rm equivalence}} & }$$
\end{definition}

\begin{remark}[Relationship between Pre-App, App, and Cofibration App] \label{rem:AppsRelationships} \leavevmode
\begin{enumerate}
\item
Clearly, App 2 implies both Pre-App 2 and Cofibration App 2.
\item \label{rem:AppsRelationships:domain_all_cofibrations_IMPLIES_Pre-App2=App2}
If every morphism of the domain is a cofibration, then conditions Pre-App 2 and App 2 coincide. If $F$ satisfies either, then $F$ also satisfies Cofibration App 2.
\item
If every morphism of the codomain is a cofibration, then conditions App 2 and Cofibration App 2 coincide. If $F$ satisfies either, then $F$ also satisfies Pre-App 2.
\item
If every morphism of both the domain and the codomain is a cofibration then conditions Pre-App 2, App 2, and Cofibration App 2 coincide.
\item
In any case, the equivalence-reflection conditions Pre-App 1, App 1, and Cofibration App 1 are all the same condition.
\end{enumerate}
\end{remark}

\begin{proposition} \label{prop:compositions_satisfy_App2}
Let $E$, $F$, and $G$ be exact functors between Waldhausen quasicategories as below.
$$\xymatrix{\underline{\cala} \ar[r]^E & \cala \ar[r]^F & \calb \ar[r]^G & \underline{\calb}}$$
Suppose $F$ satisfies Cofibration App 2.
\begin{enumerate}
\item \label{prop:compositions_satisfy_App2:hyp:FE}
If $E$ is essentially surjective, then $F\circ E$ also satisfies Cofibration App 2.
\item \label{prop:compositions_satisfy_App2:hyp:ho(coG)}
If $ho(coG)$ is both full and essentially surjective, then $G\circ F$ also satisfies Cofibration App 2.
\item \label{prop:compositions_satisfy_App2:hyp:natural_equivalence}
If an exact functor $\underline{F}$ is naturally equivalent to $F$, then $\underline{F}$ also satisfies Cofibration App 2.
\end{enumerate}
Similar statements hold if ``Cofibration App 2'' is replaced throughout by ``Pre-App 2'' and in \ref{prop:compositions_satisfy_App2:hyp:ho(coG)} the category $ho(coG)$ is replaced by $ho(G)$. \\
Similar statements also hold if ``Cofibration App 2'' is replaced throughout by ``App 2'' and in \ref{prop:compositions_satisfy_App2:hyp:ho(coG)} the category $ho(coG)$ is replaced by $ho(G)$.
\end{proposition}
\begin{proof} \leavevmode
\begin{enumerate}
\item
Suppose $E$ is essentially surjective.
Let $g\co FE\underline{a} \rightarrowtail b$ be a cofibration in $\calb$.  We factor $g$ as $F(E\underline{a}) \overset{Ff}{\rightarrowtail} F(a') \overset{w}{\to} b$. We select any equivalence $k\co a' \to E (\underline{a}')$, and obtain the desired factorization as
$$\xymatrix@C=4pc{(FE)\underline{a}\, \ar@{>->}[r]^-{F(k \circ f)} \ar@{>->}@/_1.5pc/[rr]_-g & (FE)(\underline{a}') \ar[r]^-{w \,\circ \,(Fk^{-1})}_{\simeq} & b}.$$
To find the filling 2-simplex, we work with homotopy classes in the homotopy category to arrive at a commutative triangle, and then revert back to representatives. A commutative triangle in the homotopy category always comes from a 2-simplex between any selected representatives.
\item
Suppose $ho(coG)$ is both full and essentially surjective. Let $\underline{g} \co GFa \rightarrowtail \underline{b}$ be a cofibration in $\underline{\calb}$ and select any equivalence $k\co \underline{b} \to Gb$. Let $g$ be a cofibration in $\calb$ such that $Gg=k \circ \underline{g}$. We factor $g$ as $g=w \circ Ff$ and obtain the desired factorization as
$$\xymatrix@C=4pc{(GF)a\, \ar@{>->}[r]^-{(GF)f} \ar@{>->}@/_1.5pc/[rr]_-{\underline{g}} & (GF)(a') \ar[r]^-{k^{-1} \,\circ \,Gw}_{\simeq} & \underline{b}}.$$
\item
Suppose $\underline{F}$ is exact and $\underline{F}\simeq F$, and let $g\co \underline{F}a \rightarrowtail b$ be a cofibration in $\calb$. Then we precompose $g$ with $Fa \simeq \underline{F}a$, factor the resulting morphism as $w \circ Ff$, and then use the naturality square to factor $g$ as $\underline{w} \circ \underline{F}f$ with the same $f$.
\end{enumerate}
\end{proof}

\begin{proposition}[Cof Apps for $F$ $\Rightarrow$ Cof Apps for $S_n^\infty F$ and $\overline{\mathcal{F}_{n-1}^\infty}(F)$] \label{prop:Cof_Apps_F_Imply_Cof_Apps_SnF}
If $F\co \cala \to \calb$ is an exact functor that satisfies Cofibration App 1 and 2 in Definition~\ref{def:Cof_Apps},
then both $S_n^\infty F$ and $\overline{\mathcal{F}_{n-1}^\infty}(F)$ are exact functors that satisfy Cofibration App 1 and 2.
\end{proposition}
\begin{proof}
It suffices to check that $\overline{\mathcal{F}_{n-1}^\infty}(F)$ satisfies Cofibration App 1 and 2 because the forgetful functor $S_n^\infty \calc \to \overline{\mathcal{F}_{n-1}^\infty}(\calc)$ is a Waldhausen equivalence by Proposition~\ref{prop:SF_equivalence}. Namely, conjugation of $\overline{\mathcal{F}_{n-1}^\infty}(F)$ by this natural Waldhausen equivalence, shows that $S_n^\infty F$ satisfies Cofibration App 1 when $\overline{\mathcal{F}_{n-1}^\infty}(F)$ does. An application of all three parts of Proposition~\ref{prop:compositions_satisfy_App2} to the conjugation shows that $S_n^\infty F$ satisfies Cofibration App 2 when $\overline{\mathcal{F}_{n-1}^\infty}(F)$ does.

For simplicity of notation we use $n$ instead of $n-1$. Recall that $I[n]$ is the union of the standard 1-simplices $N\{i<i+1\}$ for $0\leq i \leq n-1$, and is called the {\it spine} of $\Delta[n]$. The objects of $\overline{\mathcal{F}_{n}^\infty}(\cala)$ are functors $A\co I[n] \to co \cala$, which we abbreviate as $(A_i)_{i=0}^n$. An object of $\overline{\mathcal{F}_{n}^\infty}(\cala)$ is a sequence of cofibrations, head-to-tail, {\it without composites and without quotients}.

If $F$ satisfies Cofibration App 1, then it is clear that $\overline{\mathcal{F}_{n}^\infty}(F)$ also satisfies Cofibration App 1, because natural equivalences are precisely the natural transformations that are level-wise equivalences by \cite[Theorem 5.14]{JoyalQuadern}.

Suppose $F$ satisfies Cofibration App 2. We show that $\overline{\mathcal{F}_{n}^\infty}(F)$ satisfies Cofibration App 2. This proof is similar to Lemma 1.6.6 of Waldhausen \cite{WaldhausenAlgKTheoryI} concerning the analogous statement for App 2. We only need to justify the quasicategorical translation and check that the relevant maps are cofibrations.
Let $g \co \overline{\mathcal{F}_{n}^\infty}(F)(A)\rightarrowtail B$ be a cofibration in $\overline{\mathcal{F}_{n}^\infty}(\calb)$, and suppose for an inductive proof that $\overline{\mathcal{F}_{n-1}^\infty}(F)$ satisfies Cofibration App 2. Then our task is construct a factorization $g_n=w_n \circ F f_n$ which together with the already constructed $(w_i)_{i=0}^{n-1}$ and $(f_i)_{i=0}^{n-1}$ forms a desired factorization in $\overline{\mathcal{F}_{n}^\infty}(\calb)$. We first form the pushout $P_1$ on the left below. The middle diagram is the $F$-image of the pushout square, and the last commutative square of $g$. The universal property of the pushout $FP_1$ induces a contractible space of morphisms $FP_1\to B_n$, we select one and call it $g_n'$. In the third diagram, the outer bottom square is the bottom square of the middle diagram, while the inner bottom square is a pushout, and therefore there is an induced selected morphism $p_2$.

$$
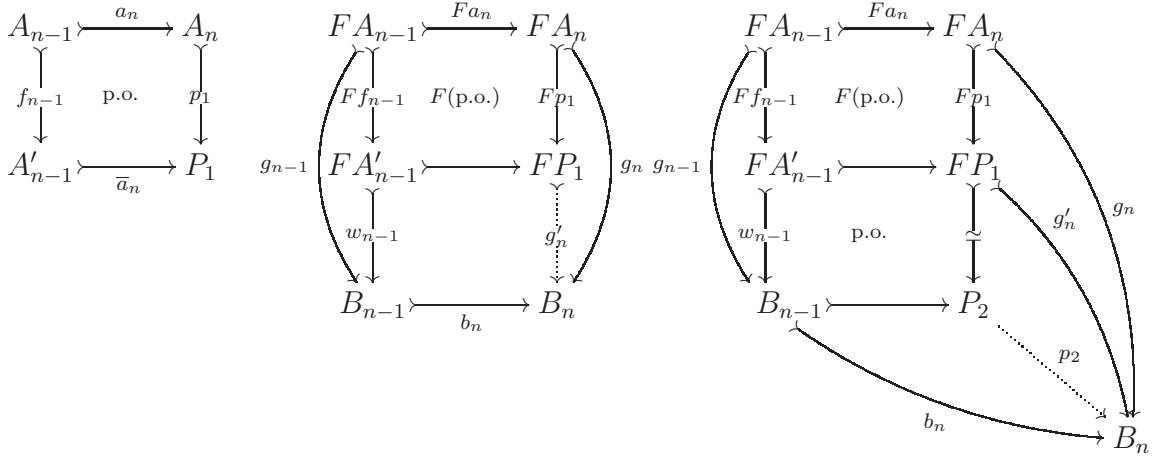
\xymatrix@C=2.8pc@R=2.8pc{
A_{n-1} \ar@{}[dr]|{\text{p.o.}} \ar@{>->}[r]^-{a_n}  \ar@{>->}[d]|{f_{n-1}} & A_n \ar@{>->}[d]|{p_1} &
FA_{n-1} \ar@{}[dr]|{F(\text{p.o.})} \ar@{>->}[r]^-{Fa_n} \ar@{>->}[d]|{Ff_{n-1}} \ar@/_1.7pc/@{>->}[dd]_{g_{n-1}} & FA_n \ar@{>->}[d]|{Fp_1} \ar@/^1.7pc/@{>->}[dd]^{g_{n}}  &
\hspace{.25in}FA_{n-1} \ar@{}[dr]|{F(\text{p.o.})} \ar@{>->}[r]^-{Fa_n} \ar@{>->}[d]|{Ff_{n-1}} \ar@/_1.7pc/@{>->}[dd]_{g_{n-1}} & FA_n \ar@{>->}[d]|{Fp_1} \ar@/^1.7pc/@{>->}[dddr]^{g_n}  & \\
A'_{n-1} \ar@{>->}[r]_-{\overline{a}_n} &  P_1 &
FA'_{n-1} \ar@{>->}[r] \ar@{>->}[d]|{w_{n-1}} &  FP_1 \ar@{>.>}[d]|{g_n'} &
\hspace{.25in}FA'_{n-1} \ar@{>->}[r] \ar@{>->}[d]|{w_{n-1}} \ar@{}[dr]|{\text{p.o.}} &  FP_1 \ar@{>->}[d]|{\simeq} \ar@/^1pc/@{>->}[ddr]^(.3){g_n'} & \\
& &
B_{n-1} \ar@{>->}[r]_-{b_n} & B_n &
\hspace{.25in}B_{n-1} \ar@{>->}[r] \ar@/_1pc/@{>->}[drr]_-{b_n} & P_2 \ar@{.>}[dr]^{p_2} & \\
& & & & & & B_n}$$

But $p_2$ is a cofibration: the two pushout squares in the right diagram compose to make a pushout square, $g$ is a cofibration in $\overline{\mathcal{F}_{n}^\infty}(\calb)$, and $p_2$ is also the induced map for the outermost square involving $g_{n-1}$ and $g_n$. We have now expressed $g_n'$ as a composite of two cofibrations, so $g_n'$ is also a cofibration.

We can now apply Cofibration App 2 to $g_n'\co FP_1 \rightarrowtail B_n$ and factor it as
$$\xymatrix@C=3pc{FP_1 \ar@{>->}[r]^-{Ff_n'} \ar@{>->}@/_1.5pc/[rr]_-{g_n'} & FA_{n}' \ar[r]^-{w_n} & B_n}$$
($A_n'$ is defined this way). We define $f_n:=f_n' \circ p_1$ and have
$$g_n=g_n' \circ Fp_1= \big(w_n \circ Ff_n' \big)\circ Fp_1=w_n \circ F(f_n' \circ p_1) = w_n \circ F f_n.$$
As remarked before, the previous sequence of equalities is actually performed in the homotopy category, and then we revert back to our selected representatives and a 2-simplex for $g_n = w_n \circ F f_n$ exists by elementary quasicategory theory.

We define the cofibration $a_n' \co A_{n-1}' \rightarrowtail A_n'$ to be $a_n':=f_n'\circ \overline{a}_n$, which is a cofibration because both $f_n'$ and $\overline{a}_n$ are.

Any map $P_1 \to A_n'$ induced by $a_n' \circ f_{n-1} = f_n \circ a_n$ is homotopic to $f_n'$  (by virtue of $a_n'$ and $f_n$ being defined as composites with second map $f_n'$), so $P_1 \to A_n'$ is a cofibration as $f_n'$ is. Thus $f=(f_i)_{i=0}^n$ is a cofibration in $\overline{\mathcal{F}_{n}^\infty}(\cala)$.
\end{proof}

\begin{remark} \label{rem:Assuming_PreApp2_Instead}
We would like an analogous statement to Proposition~\ref{prop:Cof_Apps_F_Imply_Cof_Apps_SnF} for Pre-Apps, but we will need to assume that every morphism of the domain is a cofibration. Why? Notice that the proof used in Proposition~\ref{prop:Cof_Apps_F_Imply_Cof_Apps_SnF} for Cofibration App 2 works fine for Pre-App 2 without further assumption,\footnote{After removal of the word ``cofibration'' in relevant passages.} {\it until the penultimate paragraph where $a_n'$ is proved to be a cofibration}. Assuming only Pre-App 2 instead of Cofibration App 2, we can still define $a_n':=f_n'\circ \overline{a}_n$, however it is not necessarily a cofibration, as $f_n'$ is not necessarily a cofibration. Additionally assuming that every morphism of $\cala$ is a cofibration of course guarantees that $a_n'$ is a cofibration as needed, in order for $(A_i')_{i=0}^n$ to be an element of $\overline{\mathcal{F}_{n}^\infty}(\cala)$.
\end{remark}

\begin{proposition}[Pre-Apps for $F$ and $co\cala=\cala$ $\Rightarrow$ Pre-Apps for $S_n^\infty F$] \label{prop:Pre-Apps_for_F_and_coA=A_imply_Pre-Apps_For_SnF}
Let $F\co \cala \to \calb$ be an exact functor and suppose that every morphism of $\cala$ is a cofibration. If $F$ satisfies Pre-App 1 and Pre-App 2 in Theorem~\ref{thm:Pre-Approximation}, then $S_n^\infty F$ is also an exact functor that satisfies Pre-App 1 and Pre-App 2.
\end{proposition}
\begin{proof}
Since every morphism of $\cala$ is a cofibration, every morphism of $S_n^\infty \cala$ and $\overline{\mathcal{F}_{n-1}^\infty}(F)$ is also a cofibration. Then the conditions Pre-App 2 and App 2 for $F$ coincide, as do the conditions for $S_n^\infty F$ and $\overline{\mathcal{F}_{n-1}^\infty}(F)$ (separately), by Remark~\ref{rem:AppsRelationships}~\ref{rem:AppsRelationships:domain_all_cofibrations_IMPLIES_Pre-App2=App2}.

We reduce consideration of $S_n^\infty F$ to $\overline{\mathcal{F}_{n-1}^\infty}(F)$ via Proposition~\ref{prop:compositions_satisfy_App2} for Pre-App 2 or App 2.

Waldhausen's argument of Lemma 1.6.6 in \cite{WaldhausenAlgKTheoryI} then applies, see the proof of Proposition~\ref{prop:Cof_Apps_F_Imply_Cof_Apps_SnF} and remove the word ``cofibration'' as appropriate, as stated in Remark~\ref{rem:Assuming_PreApp2_Instead}.
\end{proof}

\subsection{The Cofibration Approximation Theorem and Other Variants} \label{subsec:TheCofibrationApproximationTheorem} \leavevmode \\

We next prove a version of Approximation with Cofibration App 2 in place of App 2, and obtain various other statements as corollaries. We do not require a cylinder functor.

\begin{theorem}[Cofibration Approximation] \label{thm:CofibrationApproximation}
Let $F\co \cala \to \calb$ be an exact functor between Waldhausen quasicategories. Suppose:
\begin{enumerate}
\item \label{thm:CofibrationApproximation:F_reflects_equivs}
(Cofibration App 1) \\
$F$ reflects equivalences.
\item \label{thm:CofibrationApproximation:cofibration_approx_axiom}
(Cofibration App 2) \\
For every $a \in \cala$ and every cofibration $g\co Fa \rightarrowtail b$ in the codomain $\calb$, there exists a cofibration $f \co a \rightarrowtail a'$ in $\cala$, an equivalence $F(a') \simeq b$ in $\calb$, and a 2-simplex in $\calb$ of the following form.
$$\xymatrix{Fa \ar@{>->}[rr]^-{\forall \text{ cofibration} \; g} \ar@{>->}[dr]_{\exists \;\; F(f)} & & b \\  & F(a') \ar[ur]_{\exists \;\;\text{\rm equivalence}} & }$$
\item \label{thm:CofibrationApproximation:hoF_full_on_isos}
The functor of groupoids $ho \big(F_\text{\rm equiv}\big)\co ho \big(\cala_\text{\rm equiv}\big) \to ho \big(\calb_\text{\rm equiv}\big)$ is full, that is, if $[g]\co Fa \to Fa'$ is an isomorphism in $ho\calb$, then there exists an isomorphism $[f]\co a \to a'$ in $ho\cala$ such that $F[f]=[g]$.
\item \label{thm:CofibrationApproximation:hoFnF_full_on_isos}
$ho\big(\big(\overline{\calf^\infty_n}F\big)_\text{\rm equiv}\big)$ is also full for all $n \geq 1$.
\end{enumerate}
Then $F_{\text{\rm equiv}}$ is an equivalence of $\infty$-groupoids, as is $\big(S_n^\infty F\big)_{\text{\rm equiv}}$ for every $n \geq 0$. Consequently $\bfK(F)$ is a level-wise equivalence of spectra. \\ \\
Hypothesis \ref{thm:CofibrationApproximation:hoFnF_full_on_isos} holds for instance if $ho \big(F_\text{\rm equiv}\big)$ is fully faithful and $ho\big(\big(\overline{\calf^\infty_1} F\big)_\text{\rm equiv} \big)$ is full. See Remark~\ref{rem:sufficient_condition_for_hoF1Fequiv}.
\end{theorem}
\begin{proof}
Notice that the restriction to cofibration subquasicategories $coF\co co\cala \to co \calb$ satisfies Pre-App 1 and Pre-App 2 there because $F$ satisfies Cofibration App 1 and Cofibration App 2. We first prove that $F_{\text{\rm equiv}}$ is an equivalence of $\infty$-groupoids by applying the Pre-Approximation Theorem~\ref{thm:Pre-Approximation} to the restriction $coF$ and using the fact that $F_\text{equiv}=(coF)_\text{equiv}$. Verification of the hypotheses of Pre-Approximation Theorem~\ref{thm:Pre-Approximation} for $coF$ are as follows.
\begin{enumerate}
\item
Pre-App 1 for $coF$ is Cofibration App 1 for $F$, so holds.
\item
Pre-App 2 for $coF$ is Cofibration App 2 for $F$, so holds.
\item
The functor $ho \big((coF)_\text{\rm equiv}\big)$ is essentially surjective by Remark~\ref{rem:Pre-App2+initialobj_implies_essentially_surjective}. It is also full because $ho\big(F_\text{\rm equiv}\big)$ is full by hypothesis (recall also that every equivalence is a cofibration).
\item
The quasicategory $co\cala$ has an initial object and all pushouts, so admits all finite colimits. The functor $coF$ preserves zero objects, so also initial objects (every initial object is a zero object), and $coF$ preserves pushouts, so $coF$ preserves all finite colimits.
\end{enumerate}
Thus, by the Pre-Approximation Theorem~\ref{thm:Pre-Approximation}, the functor $(coF)_\text{equiv}\co \cala_\text{equiv} \to \calb_\text{equiv}$ is an equivalence of $\infty$-groupoids, but $(coF)_\text{equiv}=F_\text{equiv}$, so $F_\text{equiv}$ is an equivalence of $\infty$-groupoids.

We similarly prove $\big(S_n^\infty F\big)_{\text{\rm equiv}}$ is an equivalence of $\infty$-groupoids using the Pre-Approximation Theorem~\ref{thm:Pre-Approximation}. By Proposition~\ref{prop:Cof_Apps_F_Imply_Cof_Apps_SnF}, the exact functor $S_n^\infty F$ inherits Cofibration App 1 and 2 from $F$, so assumptions \ref{thm:Pre-Approximation:reflects_equivs} and \ref{thm:Pre-Approximation:approx_axiom} of Theorem~\ref{thm:Pre-Approximation} hold for $co\big(S_n^\infty F\big)$. The essential surjectivity of $ho\big( co\big(S_n^\infty F \big)_\text{equiv} \big)$ in hypothesis \ref{thm:Pre-Approximation:hoFequiv_esssurj_full} and the special colimits in hypothesis \ref{thm:Pre-Approximation:finite_posets_of_equivs_have_colimits} are also inherited from $coF$. What remains to show is the functor $ho\big(co\big(S_n^\infty F\big)_\text{equiv}\big)=ho\big(S_n^\infty F\big)_\text{equiv}$ is full in \ref{thm:Pre-Approximation:hoFequiv_esssurj_full}. But this follows from present hypothesis \ref{thm:CofibrationApproximation:hoFnF_full_on_isos} on $ho\big(\big(\overline{\calf^\infty_{n-1}}F)_\text{equiv}\big)\big)$ by Proposition~\ref{prop:SF_equivalence}, so $\big(S_n^\infty F\big)_{\text{\rm equiv}}$ is an equivalence of $\infty$-groupoids for all $n\geq2$ by the Pre-Approximation Theorem. For $n=1$ and $n=0$, the functor $\big(S_1^\infty F\big)_{\text{\rm equiv}}$ is basically $F_{\text{equiv}}$, so also an equivalence of $\infty$-groupoids, while $\big(S_0^\infty F\big)_{\text{\rm equiv}}$ is an equivalence between terminal quasicategories.

From the equivalences of $\infty$-groupoids $F_\text{equiv}$ and $\big(S_n^\infty F\big)_{\text{\rm equiv}}$, the spectrum map $\bfK(F)$ is a levelwise equivalence of spectra by the same argument as in the Approximation Theorem proof using diagram \eqref{equ:level-wise_equivalence_of_K-theory_spectra}. This completes the proof of the theorem. Next we prove the final claim about a sufficient condition for hypothesis \ref{thm:CofibrationApproximation:hoFnF_full_on_isos} to hold.

It remains to prove that hypothesis \ref{thm:CofibrationApproximation:hoFnF_full_on_isos} holds if $ho \big(F_\text{\rm equiv}\big)$ is fully faithful and $ho\big(\big(\overline{\calf^\infty_1}F\big)_\text{equiv}\big)$ is full.
To illustrate the argument it suffices to prove the claim for $n=2$. Suppose we have an equivalence $w\co\overline{\calf^\infty_2} F(A_i)_{i=0}^2 \simeq \overline{\calf^\infty_2} F(A_i')_{i=0}^2$, and consider its two commutative squares, separately. Then by the fullness of $ho\big(\big(\overline{\calf^\infty_1} F\big)_\text{equiv}\big)$, we can find two commutative squares in $\cala$
\begin{equation} \label{equ:squares_of_w}
\begin{array}{c}
\xymatrix@C=2.5pc@R=2.5pc{A_0 \ar@{>->}[r]^{c_1} \ar[d]^\simeq_{u_0} \ar[dr] & A_1 \ar[d]_\simeq^{u_1} \\ A_0' \ar@{>->}[r]_{c_1'} & A_1' }
\end{array} \hspace{1in}
\begin{array}{c}
\xymatrix@C=2.5pc@R=2.5pc{A_1 \ar@{>->}[r]^{c_2} \ar[d]^\simeq_{v_1} \ar[dr] & A_2 \ar[d]_\simeq^{v_2} \\ A_1' \ar@{>->}[r]_{c_2'} & A_2' }
\end{array}
\end{equation}
whose images are homotopic in $ho\big(\big(\overline{\calf^\infty_1}\calb\big)_\text{equiv}\big)$ to the two given squares of $w$, separately.

We would like to simultaneously alter the second square in \eqref{equ:squares_of_w} and create a homotopy between it and its alteration. Consider the prism $\Delta[1] \times \Delta[2] \to \cala$ which is the right square of \eqref{equ:squares_of_w} in the 1-face and 2-face, and is the identity on $c_2'$ in the 0-face (its base). This is the trivial ``left homotopy'' from the right square of \eqref{equ:squares_of_w} to itself. In this prism, replace the initial triangle at 0 by any 2-simplex of $\cala$ that is a left homotopy\footnote{Recall $v_1$ and $u_1$ are homotopic by the fully faithfulness of $ho(F_\text{equiv})$, since their $F$-images are both homotopic to $w_1$.} from $v_1$ to $u_1$, remove the lower 2-simplex of the prism face that now has $u_1$, and remove the (interior) 3-simplex to which this face belongs. This creates a $\Lambda^1[3]$ horn in $\cala$, which we fill, to now have a new prism $\Delta[1] \times \Delta[2] \to \cala$. This new prism is a homotopy between the right square of \eqref{equ:squares_of_w} and the following square with $u_1$ as the left vertical map, and a different lower 2-simplex.
\begin{equation} \label{equ:new_square}
\begin{array}{c}
\xymatrix@C=2.5pc@R=2.5pc{A_1 \ar@{>->}[r]^{c_2} \ar[d]^\simeq_{u_1} \ar[dr] & A_2 \ar[d]_\simeq^{v_2} \\ A_1' \ar@{>->}[r]_{c_2'} & A_2' }
\end{array}
\end{equation}

Next we glue \eqref{equ:new_square} to the left square of \eqref{equ:squares_of_w} and obtain an equivalence in $\overline{\calf_2^\infty}(\cala)$ whose $F$-image is homotopic to the equivalence $w$ from the outset. This completes the conclusion of fullness of $ho\big(\big(\overline{\calf_2^\infty} F \big)_\text{equiv}\big)$ in hypothesis \ref{thm:CofibrationApproximation:hoFnF_full_on_isos} from the indicated condition. The argument for fullness of $ho\big(\big(\overline{\calf_n^\infty} F \big)_\text{equiv}\big)$ for $n > 2$ is similar: one successively alters and pastes squares as above, one after the other.
\end{proof}

\begin{remark}[Sufficient Conditions for Fullness of $ho\big(\big(\overline{\calf^\infty_1} F\big)_\text{\rm equiv} \big)$] \label{rem:sufficient_condition_for_hoF1Fequiv}
In the final sentence of Theorem~\ref{thm:CofibrationApproximation}, a two-part sufficient condition was given for hypothesis \ref{thm:CofibrationApproximation:hoFnF_full_on_isos} to hold, the latter part of this sufficient condition was the fullness of $ho\big(\big(\overline{\calf^\infty_1} F\big)_\text{\rm equiv} \big)$. Suppose:
\begin{enumerate}
\item
$F_\text{equiv}$ is full, i.e., $F_\text{equiv}$ has the right lifting property with respect to the inclusion $\partial \Delta[1] \hookrightarrow \Delta[1]$, and
\item
$F$ has the right lifting property with respect to $\partial \Delta[2] \hookrightarrow \Delta[2]$.
\end{enumerate}
Then $\big(\overline{\calf^\infty_1} F\big)_\text{\rm equiv}$ is full, and $ho\big(\big(\overline{\calf^\infty_1} F\big)_\text{\rm equiv}\big)$ is full.
\end{remark}

Instead of requiring $F$ to induce an equivalence of homotopy categories as in Corollary~\ref{cor:Approx_hoF_equiv_and_coA=A}, it is actually sufficient to require only an equivalence of the cofibration homotopy categories, provided $ho((\overline{\calf^\infty} F)_\text{equiv})$ is full.

\begin{corollary}[Approximation, when $F$ induces an equivalence of {\it cofibration} homotopy categories] \label{cor:Approximation_when_F_induces_cof_equiv}
Let $F\co \cala \to \calb$ be an exact functor between Waldhausen quasicategories. Suppose:
\begin{enumerate}
\item
$F$ induces an equivalence of cofibration homotopy categories, i.e. the functor
$$ho(co F)\co ho(co\cala) \to ho(co\calb)$$
is an equivalence of categories.
\item \label{cor:Approximation_when_F_induces_cof_equiv:ii}
$ho\big(\big(\overline{\calf^\infty_1} F\big)_\text{\rm equiv} \big)$ is full.
\end{enumerate}
Then $F_{\text{\rm equiv}}$ is an equivalence of $\infty$-groupoids, as is $\big(S_n^\infty F\big)_{\text{\rm equiv}}$ for every $n \geq 0$. Consequently $\bfK(F)$ is a level-wise equivalence of spectra.
\end{corollary}
\begin{proof}
This proof is a direct application of the Cofibration Approximation Theorem~\ref{thm:CofibrationApproximation}. Hypotheses \ref{thm:CofibrationApproximation:F_reflects_equivs}, \ref{thm:CofibrationApproximation:cofibration_approx_axiom}, and \ref{thm:CofibrationApproximation:hoF_full_on_isos} of the Cofibration Approximation Theorem all follow from the assumption that $ho(coF)$ is an equivalence of cofibration homotopy categories.

For hypothesis \ref{thm:CofibrationApproximation:hoFnF_full_on_isos}, we use the final claim of the Cofibration Approximation Theorem about a sufficient condition for \ref{thm:CofibrationApproximation:hoFnF_full_on_isos} to hold. The functor $ho\big(F_\text{equiv})$ is fully faithful because $ho(coF)$ is an equivalence of categories. The fullness of $ho\big(\big(\overline{\calf^\infty_1} F\big)_\text{\rm equiv} \big)$ is assumed.

\end{proof}

\begin{corollary} \label{cor:cofibration_inclusion}
Let $\calc$ be Waldhausen quasicategory. Then $co\calc \hookrightarrow \calc$ induces a level-wise equivalence in $K$-theory.
\end{corollary}
\begin{proof}
This follows directly from Corollary~\ref{cor:Approximation_when_F_induces_cof_equiv}. Hypothesis~\ref{cor:Approximation_when_F_induces_cof_equiv:ii} follows because every equivalence of $\calc$ is a cofibration, and because a natural transformation is a natural equivalence if and only if each component is an equivalence by \cite[Theorem~5.14]{JoyalQuadern}.
\end{proof}

\begin{lemma} \label{lem:reflection_of_cofibrations_vs_equiv_of_cofibration_homotopy_categories}
Let $F\co \cala \to \calb$ be an exact functor between Waldhausen quasicategories such that $hoF$ is an equivalence of categories. Then the functor $F$ reflects cofibrations if and only if it induces an equivalence of cofibration homotopy categories, that is, $F$ reflects cofibrations if and only if
$$ho \left( co F \right) \co ho \left(  co \cala \right) \to ho \left(co \calb \right)$$ is an equivalence of categories.
\end{lemma}
\begin{proof}
Suppose $hoF$ is an equivalence of categories. Since $ho F$ is faithful and $ho\left(co \cala\right)$ is naturally a subcategory of $ho\cala$ that maps to $ho\left(co\calb\right)$, the exact functor $F$ induces a faithful functor of homotopy {\it cofibration} categories. Since $ho F$ is essentially surjective, and equivalences are cofibrations,  and essentially surjective functor $ho (co F) \co ho (co \cala) \to ho (co \calb)$ is also essentially surjective.

Then $F$ reflects cofibrations if and only if $ho (co F)$ is full (recall that any map homotopic to a cofibration is a cofibration). Consequently $F$ reflects cofibrations if and only if $ho (co F)$ is an equivalence of categories.
\end{proof}

\begin{corollary}[Approximation, when $F$ reflects cofibrations and induces equivalence of homotopy categories] \label{cor:Approximation_when_F_rflcts_cofs_and_hoF_is_equivalence}
Let $F\co \cala \to \calb$ be an exact functor between Waldhausen quasicategories. Suppose:
\begin{enumerate}
\item
$F$ induces an equivalence of homotopy categories
$$hoF\co ho\cala \to ho\calb.$$
\item
$ho\big(\big(\overline{\calf^\infty_1} F\big)_\text{\rm equiv} \big)$ is full.
\item
$F$ reflects cofibrations.
\end{enumerate}
Then $F_{\text{\rm equiv}}$ is an equivalence of $\infty$-groupoids, as is $\big(S_n^\infty F\big)_{\text{\rm equiv}}$ for every $n \geq 0$. Consequently $\bfK(F)$ is a level-wise equivalence of spectra.
\end{corollary}
\begin{proof}
By Lemma~\ref{lem:reflection_of_cofibrations_vs_equiv_of_cofibration_homotopy_categories}, the functor $ho(co F)$ is an equivalence of categories, so the result follows from Corollary~\ref{cor:Approximation_when_F_induces_cof_equiv}.
\end{proof}

\begin{remark} \label{rem:BlumbergMandell_DKequiv_of_weak_cof_cats}
In the classical context, a related result to Corollaries~\ref{cor:Approximation_when_F_induces_cof_equiv} and \ref{cor:Approximation_when_F_rflcts_cofs_and_hoF_is_equivalence} is the following. Let $\cala$ and $\calb$ be classical Waldhausen categories both of which have the 3-for-2 property for weak equivalences, and both of which admit functorial mapping cylinders for weak cofibrations (FMCWC) in the sense of \cite[Definition~2.6]{BlumbergMandellAbstHomTheory}. If $F \co \cala \to \calb$ is an exact functor which induces an equivalence of homotopy categories and induces a Dwyer-Kan equivalence of DK-localizations of subcategories of weak cofibrations, then it follows from \cite[Theorem~2.7]{BlumbergMandellAbstHomTheory} of Blumberg--Mandell that $F$ induces a stable equivalence in $K$-theory.
\end{remark}

\begin{remark}
In Corollary~\ref{cor:Approximation_when_F_rflcts_cofs_and_hoF_is_equivalence} we made no assumption that $F$ satisfies App 1 and App 2, but instead required reflection of cofibrations and an equivalence of homotopy categories.\footnote{When all morphisms in the domain are cofibrations, any exact functor that induces an equivalence of homotopy categories indeed does satisfy App 1 and App 2, see Corollary~\ref{cor:Approx_hoF_equiv_and_coA=A}. But in Corollary~\ref{cor:Approximation_when_F_rflcts_cofs_and_hoF_is_equivalence} we did not assume every morphism of the domain is a cofibration, so we made no assumption about App 1 and App 2 in Corollary~\ref{cor:Approximation_when_F_rflcts_cofs_and_hoF_is_equivalence}.} In the classical setting, under certain hypotheses, these two sets of assumptions are equivalent. Namely, Blumberg--Mandell prove in \cite[Theorems 1.5 and 9.1]{BlumbergMandellAbstHomTheory} that if $\cala$ and $\calb$ are both classical Waldhausen categories whose weak equivalences have the 3-for-2 property and $\cala$ admits factorization, then any exact functor $F\co \cala \to \calb$ that satisfies App 1 and App 2 induces an equivalence of homotopy categories. There are several related results in the present paper in the setting of quasicategories: App 2 implies $hoF$ is essentially surjective (see Remark~\ref{rem:Pre-App2+initialobj_implies_essentially_surjective}), App 1 and App 2 together imply $hoF$ is full (see Lemma~\ref{lem:ImageOf_hoF}), and if the domain Waldhausen quasicategory admits factorization, then App 1 and App 2 imply $F_\text{equiv}$ is an equivalence of $\infty$-groupoids (see the Approximation Theorem~\ref{thm:Approximation}).
\end{remark}

\begin{remark}
A different variant of Approximation for an exact functor $F \co \cala \to \calb$ betweeen classical Waldhausen categories was proved by Sagave in \cite[Theorem 2.8]{SagaveSpecialApproximation}. Instead of requiring Waldhausen's App 2 for all maps $Fa \to b$ in the Waldhausen category $\calb$, he requires App 2 only for maps with codomain $b$ in a full subcategory of {\it special objects}. The category $\calb$ is equipped with a functorial replacement of any object by a special one. If $\cala$ and $\calb$ are both classical Waldhausen categories whose weak equivalences have the 3-for-2 property and $\cala$ admits factorization, then any exact functor $F\co \cala \to \calb$ that satisfies special approximation and reflects weak equivalences induces an equivalence of algebraic $K$-theory spectra. An example of such a functor arises when $\cala$ and $\calb$ are full subcategories of pointed model categories, $\cala$ and $\calb$ both have full subcategories of special objects with functorial replacement, and $HoF\co Ho \cala \to Ho \calb$ is an equivalence, see \cite[Section 3.2]{SagaveSpecialApproximation}.
\end{remark}

\section{Appendix: Quasicategorical Recollections and a Criterion for a Simplicial Set to be Weakly Contractible} \label{sec:recollections}

Boardman and Vogt \cite{BoardmanVogt} originally defined the concept of quasicategory under the name {\it weak Kan complex}. Joyal \cite{JoyalQuadern} and Lurie \cite{LurieHigherToposTheory}, \cite{LurieHigherAlgebra} have extensively developed the theory of quasicategories. We rapidly review quasicategories, their homotopy categories, the join and slice of quasicategories, and colimits in a quasicategory in Sections~\ref{subsec:quasicategories_and_functors}--\ref{subsec:colimits_in_a_quasicategory}. The main reference for Sections~\ref{subsec:quasicategories_and_functors}--\ref{subsec:colimits_in_a_quasicategory} is Joyal's Barcelona notes \cite{JoyalQuadern} on quasicategories. Over quasicategories and the quasicategorical Quillen Theorem~A are recalled in Section~\ref{subsec:Quasicategorical_Thm_A}. In Section~\ref{subsec:barycentric_subdivision_and_Ex} we quickly review barycentric subdivision of simplicial sets and Kan's functor $\text{Ex}$ to prepare for a weak contractibility criterion in Section~\ref{subsec:criterion_for_weak_contractibility}. This criterion in Proposition~\ref{prop:extracted_from_Waldhausen} is an important ingredient for the central result of the present paper, the Pre-Approximation Theorem~\ref{thm:Pre-Approximation}, so we prove every detail. The argument is an adaptation of Schlichting's \cite[Lemma 14 on page 131]{Schlichting}, which he extracted from Waldhausen \cite[pages 354--356]{WaldhausenAlgKTheoryI}.

\subsection{Quasicategories and Functors} \label{subsec:quasicategories_and_functors} \leavevmode \\

A {\it quasicategory} is a simplicial set $X$ in which every inner horn admits a filler. That is, for any $0<k<n$ and any map $\Lambda^k[n] \to X$, there exists a map $\Delta[n] \dashrightarrow X$ such that the diagram
$$\xymatrix{\Lambda^k[n] \ar[r] \ar@{^{(}->}[d] & X  \\ \Delta[n] \ar@{-->}[ur]_{\exists} & }$$
commutes. The objects and morphisms of a quasicategory are, respectively, its vertices and edges. A {\it functor between quasicategories} is merely a map of simplicial sets. The quasicategory of functors from a simplicial set $A$ to a quasicategory $X$ is the usual internal hom simplicial set $X^A$. The 1-simplices of $X^A$ are the {\it natural transformations}, that is, simplicial set maps $\alpha\co A \times \Delta[1] \to X$. A functor $F \co X \to Y$ between quasicategories is {\it fully faithful} if $F(a,b)\co X(a,b) \to Y(Fa,Fb)$ is a weak homotopy equivalence, while $F$ is {\it essentially surjective} if $\tau_1F$ is essentially surjective. Here $X(a,b)$ is the mapping space recalled below, and $\tau_1\co \mathbf{SSet} \to \mathbf{Cat}$ is the left adjoint to the nerve functor $N \co \mathbf{Cat} \to \mathbf{SSet}$. A functor $F$ between quasicategories is essentially surjective if and only if every object in its codomain is equivalent to one in its image. A functor between quasicategories is called an {\it equivalence} if it is an equivalence in the 2-category $\mathbf{SSet}^{\tau_1}$ which has simplicial sets as its objects and $\mathbf{SSet}^{\tau_1}(A,B):=\tau_1(B^A)$ as its hom categories. A functor between quasicategories is an equivalence if and only if it is fully faithful and essentially surjective \cite[Propositions~3.19 and 3.20]{FiorePieper}.

A sub simplicial set $A$ of a quasicategory $X$ is called {\it full} or {\it 0-full} if any simplex of $X$ is in $A$ if and only if all of its vertices are in $A$. Any 0-full sub simplicial set of a quasicategory is automatically a quasicategory \cite[page 275]{JoyalQuadern}. A sub quasicategory $W$ of a quasicategory $X$ is called {\it 1-full} if any simplex of $X$ is in $W$ if and only if all of its edges are in $W$.

\subsection{Homotopy and Equivalence in a Quasicategory} \label{subsec:homotopy_and_equivalence} \leavevmode \\

Boardman and Vogt associated to any quasicategory $X$ its homotopy category $hoX$. A 1-simplex of a quasicategory $X$ is called a {\it morphism}. Two parallel morphisms $f,g\co a \to b$ are {\it left homotopic} if there exists a 2-simplex $\sigma \in X_2$ with boundary $\partial \sigma=(d_0\sigma,d_1\sigma,d_2\sigma)=(1_b,g,f)$. They are {\it right homotopic} if there exists a 2-simplex $\sigma \in X_2$ with boundary $\partial \sigma=(g,f,1_a)$. They are {\it homotopic} if they are in the same path component of the {\it mapping space} $X(a,b)$, which is the following pullback.
\begin{equation} \label{equ:mapping_space_of_a_quasicategory}
\begin{array}{c}
\xymatrix{X(a,b) \ar[r] \ar[d] \ar@{}[dr]|{\text{pullback}} & X^{\Delta[1]} \ar[d]^{(s,t)} \\ \ast \ar[r]_-{(a,b)} & X \times X }
\end{array}
\end{equation}
All three notions of homotopy coincide and are an equivalence relation on the morphisms of the quasicategory $X$. The 0-simplices of $X$ together with the homotopy classes of morphisms form the {\it homotopy category} $hoX$ \cite{BoardmanVogt}. This category is isomorphic to the {\it fundamental category} $\tau_1X$, so the categories $hoX$ and $\tau_1X$ are identified without further mention when $X$ is a quasicategory. Here $\tau_1$ is the left adjoint to the nerve functor.

A morphism in $X$ is an {\it equivalence} if its homotopy class is an isomorphism in $hoX$, which is the case if and only if the morphism has a ``homotopy inverse'' in $X$. We denote the 1-full subquasicategory of $X$ on the equivalences by $X_\text{equiv}$. This is the maximal Kan subcomplex of $X$, and $(-)_\text{equiv}$ is right adjoint to the inclusion $\mathbf{Kan} \hookrightarrow \mathbf{SSet}$ \cite[Theorems 4.18 and 4.19]{JoyalQuadern}. If $X$ and $Y$ are quasicategories, then a natural transformation $\alpha\co X \times \Delta[1] \to Y$ is an equivalence in the quasicategory $Y^X$ if and only if each component $\alpha_x=\alpha(x,-)$ is an equivalence in $Y$, see \cite[Theorem 5.14]{JoyalQuadern}.

\begin{example}[Nerve of a category]
The nerve of any category $\calc$ is a quasicategory. The homotopy class of a morphism $f$ in $N\calc$ is simply $\{f\}$. The mapping space $(N\calc)(a,b)$ is $\calc(a,b)$ viewed as a discrete simplicial set. Consequently, the homotopy category of $N\calc$ is just $\calc$, and a morphism in $N\calc$ is an equivalence if and only if it is an isomorphism in $\calc$. Both of these consequences also follow from the fact that $\tau_1N\cong \text{Id}_\mathbf{Cat}$ (nerve is fully faithful). As expected, we now see as a special case of the above-mentioned \cite[Theorem 5.14]{JoyalQuadern}, the classical lemma that a natural transformation $\calc \times [1] \to \cald$ has components isomorphisms if and only if it is invertible in $\mathbf{Cat}(\calc,\cald)$.
\end{example}

\subsection{Join and Slice for Categories and for Simplicial Sets} \label{subsec:join_and_slice} \leavevmode \\

Recall that if $\cala$ and $\calb$ are ordinary categories, then the {\it join} $\cala \star \calb$ is the category with objects and morphisms

\begin{equation} \label{equ:join_of_categories_objects}
\ob \big( \cala \star \calb \big):=\ob (\cala) \sqcup \ob (\calb)
\end{equation}
\begin{equation} \label{equ:join_of_categories_morphisms}
\mor\big(\cala \star \calb\big):=\mor(\cala) \sqcup \mor(\calb) \sqcup \left( \underset{a \in \ob \cala, \; b \in \ob \calb }\bigsqcup \{f_{a,b}\co a \to b   \} \right).
\end{equation}
For, $a,a'\in \ob \cala$ and $b,b' \in \ob \calb$, any composite $\xymatrix@1@C=1.7pc{a' \ar[r] & a \ar[r]^{f_{a,b}} & b \ar[r] & b'}$ is the unique map $a' \to b'$, namely $f_{a',\,b'}$. All other possible compositions are already defined. For an example of a familiar join, notice that the join of any $(m+1)$-element linearly ordered set with any $(n+1)$-element linearly ordered set is an $(m+n+2)$-element linearly ordered set, so $[m]\star[n]=[m+n+1]$. If 1 denotes the terminal category, then $\cala \star 1$ is $\cala$ with a terminal object adjoined, while $1 \star \calb$ is $\calb$ with an initial object adjoined.

For any simplicial sets $A$ and $B$, the {\it join} simplicial set $A\star B$ of \cite[page 243]{JoyalQuadern} satisfies
\begin{equation} \label{equ:join_definition_quasicategories}
(A \star B)_n=A_n \sqcup B_n \sqcup \Bigg(\,\underset{i+1+j=n}{\bigsqcup}  A_i \times B_j \Bigg)
\end{equation}
by \cite[Proposition 3.2]{JoyalQuadern}. Compare equation \eqref{equ:join_definition_quasicategories} for $n=0$ and $n=1$ with the object and morphism formulas for the join of categories in equations \eqref{equ:join_of_categories_objects} and \eqref{equ:join_of_categories_morphisms} above. The face and degeneracy maps can be understood from the case where $A$ and $B$ are nerves of categories.
For this paper, we may take \eqref{equ:join_definition_quasicategories} as a definition. The empty simplicial set is a unit for join, $A \star \emptyset \cong A \cong \emptyset \star A$.

Clearly, $A \subseteq A\star B$ and $B \subseteq A\star B$, and there are two functors
$$\xymatrix{A \star (-)\co \mathbf{SSet} \ar[r] & (A \downarrow \mathbf{SSet}) }$$
$$\xymatrix{(-) \star B \co \mathbf{SSet} \ar[r] & (B \downarrow \mathbf{SSet}). }$$
These each admit a right adjoint called {\it slice} \cite[Proposition 3.12]{JoyalQuadern}, we denote their respective right adjoints as $a\backslash X$ and $X/b$ respectively for $a\co A \to X$ and $b \co B \to X$ in $\mathbf{SSet}$.  In particular, there are bi-natural bijections
\begin{equation} \label{equ:join_cocone_adjunction}
(A \downarrow \mathbf{SSet})\left(\begin{array}{c} \xymatrix{A \ar@{^{(}->}[d]
\\ A \star B }
\end{array}, \begin{array}{c} \xymatrix{A \ar[d]^a
\\ X }
\end{array} \right)\cong \mathbf{SSet}\left(B, a\backslash X \right)
\end{equation}
\begin{equation} \label{equ:join_cone_adjunction}
(B \downarrow \mathbf{SSet})\left(\begin{array}{c} \xymatrix{B \ar@{^{(}->}[d]
\\ A \star B }
\end{array}, \begin{array}{c} \xymatrix{B \ar[d]^b
\\ X }
\end{array} \right)\cong \mathbf{SSet}\left(A,  X/b \right).
\end{equation}
From these, we see an $n$-simplex $\Delta[n] \to a \backslash X$ is a map $A \star \Delta[n] \to X$ which extends $a$ along $A \subseteq A \star \Delta[n]$, while an $n$-simplex $\Delta[n] \to X/b$ is a map $\Delta[n] \star B \to X$ which extends $b$ along $B \subseteq  \Delta[n] \star B$.

When $X$ is a quasicategory, so are the slices $a \backslash X$ and $X /b$,
see \cite[Corollary 3.20, page 256]{JoyalQuadern} for the case of $X / b$.

The join and slice of simplicial sets is compatible with the join and slice of categories. The nerve functor sends the join of two categories to the join of their nerves \cite[Corollary 3.3]{JoyalQuadern}. So for instance, we may prove $\Delta[m]\star\Delta[n]\cong\Delta[m+n+1]$ by the sequence of isomorphisms
$$\Delta[m]\star\Delta[n]\cong N([m]\star[n])\cong N([m+n+1]) =\Delta[m+n+1].$$
The nerve preserves the slice operations \cite[Proposition 3.13]{JoyalQuadern}.
$$N(a\backslash \calc) \cong a\backslash N(\calc) \;\;\;\;\;\; N(\calc/b) \cong N(\calc)/b $$

The relationship between topological joins, categorical joins, and simplicial joins has been worked out by Fritsch--Golasi{\'n}ski \cite{FritschGolasinskiJoins}.

\subsection{Colimits in a Quasicategory} \label{subsec:colimits_in_a_quasicategory} \leavevmode \\

Joyal defined the notion of colimit in a quasicategory $X$ using join and slice as follows \cite[Definition 4.5]{JoyalQCatsAndKanComplexes}, see also \cite[page 159]{JoyalQuadern} and \cite[pages 46-49]{LurieHigherToposTheory}. An object $i$ in a quasicategory $X$ is {\it initial} if for every object $x$ in $X$ the map $X(i,x) \to \text{pt}$ is a weak homotopy equivalence. Any two initial objects of $X$ are equivalent, in the sense that there is an equivalence from one to the other. Clearly, this equivalence is even homotopically unique.
If $A$ is a simplicial set and $a\co A \to X$ is a map of simplicial sets, then a {\it cocone with base $a$} is a 0-simplex of $a\backslash X$. In other words, a cocone with base $a$ is a map $A \star 1 \to X$ which extends $a \co A \to X$ along $A \subseteq A \star 1$. A {\it colimiting cocone for $a$} is an initial object of $a\backslash X$, in other words a cocone $\eta \co A \star 1 \to X$ with base $a$ which is initial. A {\it colimit of $a$} is the value of a colimiting cocone $A \star 1 \to X$ at the unique vertex of the terminal simplicial set $1$.

If $\calc$ is a category, then all of these concepts in $N \calc$ coincide with the usual 1-category notions in $\calc$ because nerve commutes with join and slice, nerve is fully faithful, and $(N\calc)(a,b)$ is $\calc(a,b)$ viewed as a discrete simplicial set.

I thank David Gepner for sketching the following proposition to me. It will be used in Proposition~\ref{prop:SF_equivalence}, the equivalence of $S_n^\infty\calc$ with $\overline{\calf_{n-1}^\infty} \calc$.

\begin{proposition} \label{prop:colimiting_cocone_equivalence}
Let $A$ be a simplicial set, $X$ a quasicategory which admits $A$-shaped colimits, and $\mathrm{ColimCocones}\big(X^{A \star 1}\big)$ the subquasicategory of $X^{A \star 1}$ that is 0-full on the colimiting cocones. Then the restriction map $$\xymatrix{r\co \mathrm{ColimCocones}\big(X^{A \star 1}\big) \ar[r] & X^A}$$ is an equivalence of quasicategories.
\end{proposition}

\subsection{Overquasicategories and Quasicategorical Quillen Theorem A} \label{subsec:Quasicategorical_Thm_A} \leavevmode \\

For the Pre-Approximation Theorem~\ref{thm:Pre-Approximation} we need overquasicategories and the quasicategorical Quillen Theorem~A.

\begin{definition}[Overquasicategory $(Y \downarrow y)$]
Let $y$ be an object of a quasicategory $Y$. An {\it $n$-simplex over $y$} is an $(n+1)$-simplex $z \in Y_{n+1}$ such that $(\iota_{n+1})^*(z)=y$, where $\iota_{n+1}\co [0] \to [n+1]$ is $\iota_{n+1}(0)=n+1$. The simplices over $y$ form the {\it overquasicategory} $(Y \downarrow y)$. The {\it projection}
\begin{equation} \label{equ:over-quasicategory_projection}
\xymatrix{q\co (Y \downarrow y) \ar[r] & Y}
\end{equation}
is $z \mapsto d_{n+1}z$.
\end{definition}

\begin{remark} \label{rem:remarks_on_over-quasicategory}
The overquasicategory $(Y \downarrow y)$ is isomorphic to Joyal's slice $Y/y$ recalled in Section~\ref{subsec:join_and_slice}. An $n$-simplex of $Y/y$ is a map $z \co \Delta[n] \star \Delta[0] \to Y$ which extends $y \co \Delta[0] \to Y$, but the domain is $\Delta[n] \star \Delta[0] \cong \Delta[n+1]$, so such a map $z$ is the same as a map $z\co \Delta[n+1] \to Y$ with $z \circ \Delta[\iota_{n+1}]=y$. Under this isomorphism $Y/y \cong (Y \downarrow y)$, the restriction of $z \co \Delta[n]\star\Delta[0] \to Y$ to $\Delta[n]$ corresponds to the projection \eqref{equ:over-quasicategory_projection}.

The simplicial set $Y\downarrow y$ is a quasicategory when $Y$ is, by \cite[Corollary 3.20, page 256]{JoyalQuadern}. Lurie writes $Y_{/ y}$ for $(Y \downarrow y)$ and $Y/y$, see \cite[pages 42-43]{LurieHigherToposTheory}. The nerve preserves the slice operations, so if $\calc$ is a category, then $N(\calc \downarrow c)\cong(N\calc \downarrow c)$, see  \cite[Proposition 3.13, page 250]{JoyalQuadern}.
\end{remark}

\begin{definition}[Over simplicial set $(G \downarrow y)$] \label{def:(f_downarrow_y)}
Let $X$ be a simplicial set, $y$ an object of a quasicategory $Y$, and $G\co X \to Y$ a map of simplicial sets. An {\it $n$-simplex $G$-over $y$} is a pair $(x \in X_n, z \in (Y \downarrow y)_n)$ such that $G(x)=q(z)$. The simplices $G$-over $y$ form the {\it over simplicial set} $(G \downarrow y)$, it is the pullback
$$\xymatrix@C=3pc@R=3pc{(G \downarrow y) \ar[r] \ar[d] \ar@{}[dr]|{\text{pullback}} & X \ar[d]^G \\ (Y \downarrow y) \ar[r]_-q & Y.}$$
\end{definition}

In the notation of Lurie, $(G \downarrow y)$ would be $X \times_Y Y_{/y}$, see \cite[page 236]{LurieHigherToposTheory}. The simplicial set $(G \downarrow y)$ is like a {\it homotopy fiber}. The {\it left fiber} $G/(0,y)$ on page 337 of Waldhausen \cite{WaldhausenAlgKTheoryI}, which is the pullback of $G\co X \to Y$ along $y\co \Delta[0] \to Y$, is like a {\it strict fiber}.

\begin{theorem}[Quasicategorical Quillen Theorem A, \cite{JoyalI}] \label{thm:Quasicategorical_Quillen_Theorem_A}
Let $Y$ be a quasicategory, $X$ a simplicial set, and $G \co X \to Y$ a map of simplicial sets. If the over simplicial set $(G \downarrow y)$ is weakly contractible for every object $y$ of $Y$, then $G$ is a weak homotopy equivalence.
\end{theorem}
For a proof of the analogous version of Theorem~\ref{thm:Quasicategorical_Quillen_Theorem_A} for $(y \downarrow G)$, see Proposition 4.1.1.3.(3) and Theorem 4.1.3.1 of Lurie on pages 223 and 236 of \cite{LurieHigherToposTheory}. For another variant of Quillen's Theorem~A in the setting of quasicategories, see Heuts--Moerdijk \cite[Proposition~G]{HeutsMoerdijk_1}.

\subsection{Barycentric Subdivision and Kan's Functor \text{Ex}} \label{subsec:barycentric_subdivision_and_Ex} \leavevmode \\

In Proposition~\ref{prop:extracted_from_Waldhausen} we briefly need some properties of the adjunction $\text{Sd} \dashv \text{Ex}$ on $\mathbf{SSet}$, so we review it here. The {\it subdivision of the standard $n$-simplex}, denoted $\text{Sd}\,\Delta[n]$, is the nerve of the poset of non-degenerate simplices of $\Delta[n]$. In other words, $\text{Sd}\,\Delta[n]$ is the nerve of the poset of non-empty subsets of $\{0, \dots, n\}$, denoted $\big(\mathcal{P}[n]\big)\backslash \{\emptyset\}$. The {\it subdivision of a simplicial set $X$} is the colimit
$$\text{Sd}\,X:=\underset{n,\Delta[n]\to X}{\text{colim}}\,\text{Sd}\,\Delta[n]$$
where the colimit is over the {\it category of simplices of $X$}. The subdivision functor $\text{Sd} \colon \mathbf{SSet} \to \mathbf{SSet}$ is left adjoint to {\it Kan's functor} $\text{Ex}\colon \mathbf{SSet} \to \mathbf{SSet}$, as one can infer from the definition of its $n$-simplices.
$$\left(\text{Ex}\,Y\right)_n:=\mathbf{SSet}(\text{Sd}\,\Delta[n],Y)$$
For a map $g$ of simplicial sets, $\text{Ex}\,g$ is given by postcomposition with $g$.

The {\it $n$-th last vertex map} $\ell[n] \colon \text{Sd}\,\Delta[n] \to \Delta[n]$ is the nerve of the functor
$$\xymatrix{\big(\mathcal{P}[n]\big)\backslash \{\emptyset\} \ar[r] & [n]}$$
$$\xymatrix{\{v_0 < v_1 < \cdots < v_k \} \ar@{|->}[r] & v_k}.$$
The $n$-th last vertex map $\ell[n]$ is a simplicial homotopy equivalence, and the last vertex maps all together induce via precomposition a natural weak homotopy equivalence
\begin{equation} \label{equ:Y_to_ExY}
\begin{array}{c}
\xymatrix{\ell^* \colon Y \ar[r] & \text{Ex}\,Y} \\
\xymatrix{\big(\Delta[n] \to Y \big) \ar@{|->}[r] & \big(\text{Sd}\,\Delta[n] \overset{\ell[n]}{\to} \Delta[n] \overset{y}{\to} Y \big) }
\end{array}
\end{equation}
If $Y$ is based, then \eqref{equ:Y_to_ExY} is a basepoint-preserving map.  See \cite[pages 182-188]{GoerssJardine} for a recent treatment of the foregoing topics.

\subsection{A Criterion for a Simplicial Set to be Weakly Contractible} \label{subsec:criterion_for_weak_contractibility} \leavevmode \\

The next proposition is a simplicial version of the categorical result \cite[Lemma 14 on page 131]{Schlichting} by Schlichting, which he extracted from Waldhausen's paper \cite[pages 354--356]{WaldhausenAlgKTheoryI}. This proof is an adaptation of Schlichting's extraction to the situation of $\calc$ a simplicial set (instead of a category), with a few more details. This proposition is a key ingredient in the proof of the Pre-Approximation Theorem~\ref{thm:Pre-Approximation}.

\begin{proposition} \label{prop:extracted_from_Waldhausen}
Let $\calc$ be a nonempty, connected simplicial set. If for every connected finite poset $\calp$, every map $N\calp \to \calc$ can be extended to a map $(N\calp)\star 1 \to \calc$, then $\calc$ is weakly contractible.
\end{proposition}
\begin{proof}
We show that each homotopy group of the realization $\vert \calc \vert$ is trivial.

For $n=0$, we have $\pi_0|\calc|=\ast$ because $\calc$ is a connected\footnote{Recall that a simplicial set is {\it connected} if any two vertices can be connected by a zig-zag path of edges.} simplicial set by assumption, which implies that $\vert \calc \vert$ is path-connected.

Let $n \geq 1$, and let $\text{Ex}^{\infty}\co \mathbf{SSet} \to \mathbf{SSet}$ be the fibrant replacement functor of Kan reviewed in Section~\ref{subsec:barycentric_subdivision_and_Ex}.
Fix a vertex $c$ of $\calc$, and let $S^n_{\text{simp}}$ be any simplicial model of the $n$-sphere with only finitely many non-degenerate simplices. In the following, we make use of the isomorphisms
$$\aligned \pi_n(\vert \calc\vert,c)\overset{\text{def}}{=}\big[\,\vert S^n_\text{simp}\vert,\vert \calc\vert\,\big]_{\text{based}} & \overset{(1)}{\cong} \big[\,\vert S^n_\text{simp}\vert,\vert \text{Ex}^\infty\,\calc\vert\,\big]_{\text{based}} \\
& \overset{(2)}{\cong} \big[S^n_{\text{simp}}, \text{Ex}^\infty\,\calc\big]_{\text{based}}\overset{\text{def}}{=}\pi_n(\text{Ex}^\infty\,\calc,c),
\endaligned$$
where isomorphism (1) arises from the realization of the weak homotopy equivalence $\calc \to \text{Ex}^\infty\,\calc$, and isomorphism (2) is the identification of the simplicial homotopy groups of the fibrant simplicial set $\text{Ex}^\infty\,\calc$ with the standard homotopy groups of its realization. The inverse of isomorphism (2) is induced by geometric realization.

Consider a homotopy class in $\pi_n(\vert \calc\vert,c)$ and let $\alpha\co S^n_{\text{simp}}\to \text{Ex}^\infty\,\calc$ represent it via the composite of isomorphisms (1) and (2) above. We would like to show that $\vert \alpha \vert$ is based homotopic to constant $c$ to conclude the original homotopy class in $\pi_n(\vert \calc\vert,c)$ is zero via isomorphism (1).

Since $\mathbf{SSet}(S^n_{\text{simp}},-)$ commutes with directed colimits and $\text{Ex}^\infty\,\calc$ is the colimit of the directed diagram of weak homotopy equivalences
$$\xymatrix{\calc \ar[r] & \text{Ex}\,\calc \ar[r] & \text{Ex}^2\,\calc \ar[r] & \text{Ex}^3\,\calc \ar[r] & \cdots}$$
(see Kan \cite[Section 4]{KanEx} or Goerss--Jardine \cite[page 188]{GoerssJardine}),
the basepoint preserving map $\alpha\co S^n_{\text{simp}}\to \text{Ex}^\infty\,\calc$ factors through some $\beta\co S^n_{\text{simp}}\to \text{Ex}^k\,\calc$. More concretely, for each non-degenerate simplex $e$ of $S^n_{\text{simp}}$, pick an index $\ell(e)$ such that $\alpha(e)$ is in the image of $\text{Ex}^{\ell(e)}\,\calc$ in $\text{Ex}^\infty\,\calc$, and then let $k=2 +\text{max}_e \;\ell(e)$. Realizing, we also have the factorization
$$\xymatrix{\vert S^n_{\text{simp}}\vert \ar[r]_{|\beta|} \ar@/^1.6pc/[rr]^{\vert \alpha \vert} & \vert \text{Ex}^k\,\calc \vert \ar[r] & \vert \text{Ex}^\infty\,\calc\vert}.$$

We claim that it suffices to show $\vert\beta\vert$ is unbased homotopic to a constant map. For if this is the case, then the based map $\vert \alpha \vert$ is also unbased homotopic to a constant map, which implies $\vert \alpha \vert$ is based homotopic to $c$ by \cite[Corollary~7.3]{BredonBook}, in other words implies $\vert \alpha \vert$ is zero in  $[\,\vert S^n_\text{simp}\vert,\vert \text{Ex}^\infty\,\calc\vert\,]_{\text{based}}$.

So we show $\vert\beta\vert$ is unbased homotopic to a constant map. The barycentric subdivision $\text{Sd}\co \mathbf{SSet} \to \mathbf{SSet}$ is left adjoint to $\text{Ex}$ (see Kan \cite[Section 7]{KanEx} or Goerss--Jardine \cite[page 183]{GoerssJardine} or Fritsch--Piccinini \cite[4.6]{FritschPiccinini}), so $\beta\co S^n_{\text{simp}} \to \text{Ex}^k\, \calc$ corresponds to its transpose $\beta^\dagger \co \text{Sd}^k \,S^n_{\text{simp}} \to \calc$. Since $k \geq 2$, the domain finite simplicial set $\text{Sd}^k\,S^n_{\text{simp}}$ is the nerve of a finite poset (see Thomason \cite{Thomason}). This poset must be connected, as $\text{Sd}^k\,S^n_{\text{simp}}$ is connected.  By hypothesis, $\beta^\dagger$ now extends to $\overline{\beta^\dagger}\co \left( \text{Sd}^k S^n_{\text{simp}} \right) \star 1 \to \calc $. The geometric realization $\vert \text{Sd}^k S^n_{\text{simp}} \star 1 \vert$ is homeomorphic to an $(n+1)$-ball, so $\vert \beta^\dagger\vert$ is unbased homotopic to any constant in the image of this $(n+1)$-ball.

To next see that $|\beta|$ is unbased homotopic to a constant map, consider the following two commutative squares of continuous maps, which arise as the geometric realization of commutative squares of simplicial maps.
$$\xymatrix@C=4pc@R=3pc{ |S^n_{\text{simp}}| \ar[r]^-{|\beta|} \ar[d]_{| \text{unit} |} & |\text{Ex}^k\, \calc | \ar@{=}[d] \\ |\text{Ex}^k\, \text{Sd}^k \, S^n_{\text{simp}}| \ar[r]^-{|\text{Ex}^k \, \beta^\dagger|} \ar@{.>}@/_2.5pc/[d]_{(\text{h.e.})^{-1}} & |\text{Ex}^k \, \calc|  \\ |\text{Sd}^k \, S^n_{\text{simp}}| \ar[u]^{\eqref{equ:Y_to_ExY}}_{\text{h.e.}} \ar[r]_-{|\beta^\dagger|} & |\calc| \ar[u]^{\text{h.e.}}_{\eqref{equ:Y_to_ExY}} }$$
The top square is from the definition of transpose $\beta^\dagger$ for the adjunction $\text{Sd}^k \dashv \text{Ex}^k$, while the bottom square comes from the naturality square for iteration of the weak homotopy equivalence $Y \to \text{Ex}\, Y$ in \eqref{equ:Y_to_ExY}. The dotted map is any homotopy inverse to its adjacent map.

Tracing the left, bottom, and right, we now have a homotopy
$$|\beta| \simeq  (\text{h.e.}) \circ |\beta^\dagger| \circ (\text{h.e.})^{-1} \circ |\text{unit}|.$$
But since, $|\beta^\dagger|$ is unbased homotopic to a constant, so is the right-hand side, and so is $|\beta|$.

Thus $|\alpha|$ is based homotopic to constant $c$ and $\pi_n(|\calc|,c)$ is trivial for all $n \geq 0$.
\end{proof}

\section*{Acknowledgements}

{\bf Scientific Acknowledgements.} I thank Andr{\'e} Joyal for his course on quasicategories at the Centre de Recerca Matem\`atica (Barcelona) during the thematic programme on Homotopy Theory and Higher Categories in academic year 2007/08. The organizers of that glorious year have my gratitude:  Carles Casacuberta, Andr{\'e} Joyal, Joachim Kock, Amnon Neeman, and Frank Neumann.
There I learned the quasicategory theory used in this article, while supported on grant SB2006-0085 of the Spanish
Ministerio de Educaci\'{o}n y Ciencia at the Universitat Aut\`{o}noma de Barcelona.

I especially thank Andrew Blumberg for his excellent comments, and for answering my questions concerning \cite{BlumbergGepnerTabuadaI}, \cite{BlumbergMandellAbstHomTheory}, and the general topic of Approximation. His critical suggestion to consider the cofibration subcategories more closely led me to discover Corollary~\ref{cor:Approximation_when_F_induces_cof_equiv}. I thank Georgios Raptis for occasional conversations, and especially for suggesting and discussing Lemma~\ref{lem:ImageOf_hoF}.
I also thank David Gepner for previous discussions on $K$-theory and quasicategories. I thank Denis-Charles Cisinski for helpful comments and suggestions in July 2017.

I thank the Regensburg {\it Sonderforschungsbereich 1085: Higher Invariants} and the Regensburg Mathematics Department for a very stimulating working environment during during my September 2015 -- July 2016 sabbatical visit, and during his short visit July 3 -- July 16, 2017.

{\bf Financial Acknowledgements.} I gratefully acknowledge support from several sources during the genesis of this paper. A Humboldt Research Fellowship for Experienced Researchers supported me during my September 2015 -- July 2016 sabbatical at Universit{\"a}t Regensburg. The Max-Planck-Institut f\"ur Mathematik supported my research stays in Bonn in May -- June 2011 and July 2013. A Small Grant for Faculty Research from the University of Michigan-Dearborn supported some travel costs to Bonn. The Regensburg {\it Sonderforschungsbereich 1085: Higher Invariants} supported a short visit July 3 -- July 16, 2017 in Regensburg.




\end{document}